\documentclass[a4paper,11pt]{article}

\usepackage[T1]{fontenc}
\usepackage{lmodern}

\usepackage{dsfont}

\usepackage[latin1]{inputenc}
\usepackage{amsmath}
\usepackage{amsthm}
\usepackage{amssymb}
\usepackage{mathrsfs}
\usepackage{graphicx}
\usepackage[all]{xy}
\usepackage{hyperref}

\usepackage{stmaryrd}
\usepackage{caption}

\usepackage{abstract} 

\newtheorem{thm}{Theorem}[section]
\newtheorem{cor}[thm]{Corollary}
\newtheorem{claim}[thm]{Claim}
\newtheorem{fact}[thm]{Fact}

\newtheorem{lemma}[thm]{Lemma}
\newtheorem{prop}[thm]{Proposition}

\theoremstyle{definition}
\newtheorem{definition}[thm]{Definition}

\newtheorem{remark}[thm]{Remark}
\newtheorem{question}[thm]{Question}

\title{Median sets of isometries in CAT(0) cube complexes and some applications}
\date{\today}
\author{Anthony Genevois}

\begin{document}

\maketitle

\begin{abstract}
In this article, we associate to isometries of CAT(0) cube complexes specific subspaces, referred to as \emph{median sets}, which play a similar role as minimising sets of semisimple isometries in CAT(0) spaces. Various applications are deduced, including a cubulation of centralisers, a splitting theorem, a proof that Dehn twists in mapping class groups must be elliptic for every action on a CAT(0) cube complex, a cubical version of the flat torus theorem, and a structural theorem about polycyclic groups acting on CAT(0) cube complexes.
\end{abstract}

\tableofcontents

\newpage

\section{Introduction}

In the last decades, CAT(0) cube complexes have become fruitful tools in the study of groups. The reason of this success is twofold. First, many groups of interest turn out to act non-trivially on CAT(0) cube complexes, including many Artin groups, many 3-manifold groups, Coxeter groups, many small cancellation groups, one-relator groups with torsion, many free-by-cyclic groups, random groups and some Burnside groups. And second, powerful tools are now available to answer various questions about cube complexes. As a consequence, looking for an action on a CAT(0) cube complex in order to study a given group turns out to be a good strategy in order to find valuable information on it.

In this article, we are interested in the structure of subgroups of groups acting on CAT(0) cube complexes. In finite dimension, generalising the Tits alternative proved in \cite{CubicalTits}, it follows from the combination of \cite{MR2827012} and \cite{CFI} that, if a group acts properly on a finite-dimensional CAT(0) cube complex, then any of its subgroups either contains a non-abelian free subgroup or is virtually (locally finite)-by-(free abelian). Examples of groups in the second category include for instance wreath products $F \wr \mathbb{Z}^n$ where $F$ is a finite group and $n \geq 1$ \cite[Proposition 9.33]{Qm} and Houghton groups \cite{LeeHoughton}, \cite[Example 4.3]{FarleyHughes}. However, if we allow infinite-dimensional cube complexes, the dichotomy no longer holds. For instance, Thompson's group $F$ acts on an infinite-dimensional CAT(0) cube complex \cite{MR1978047} but it does not contain any non-abelian free group and it is not virtually free abelian (in fact, it is even not virtually solvable). 

There is still a lot to understand about groups acting on infinite-dimensional CAT(0) cube complexes. The present work is motivated by the following question: 

\begin{question}\label{question:intro}
Which solvable groups act properly on (infinite-dimensional) CAT(0) cube complexes? 
\end{question}

Results in this direction can be found from the general study of groups acting on CAT(0) spaces; see \cite[Chapter II.7]{MR1744486}. For instance, a polycyclic group which acts properly on a CAT(0) space by semisimple isometries must be virtually abelian \cite[Theorem II.7.16]{MR1744486}. Such statements are fundamentally based on the description of \emph{minimising sets} of semisimple isometries: their product structures and their compatibility with normalisers. See \cite[Chapters II.6 and II.7]{MR1744486}. However, this strategy only applies to actions by semisimple isometries, and it turns out that groups may act on infinite-dimensional CAT(0) cube complexes with parabolic isometries, like Thompson's group $F$ \cite{ParabolicF, ThompsonCAT}. Following the metatheorem saying that the statements which hold for actions on CAT(0) spaces by semisimple isometries also hold for actions on CAT(0) cube complexes without assumptions on the type of isometries, the main goal of this article is to identify the subsets in CAT(0) cube complexes which play the role of minimising sets of semisimple isometries in CAT(0) spaces. 

Following \cite{HaglundAxis}, isometries of CAT(0) cube complexes can be classified in three mutually exclusive categories:
\begin{itemize}
	\item \emph{elliptic isometries}, for isometries which stabilise cubes;
	\item \emph{loxodromic isometries}, for isometries preserving bi-infinite (combinatorial) geodesics;
	\item \emph{inverting isometries}, for isometries with unbounded orbits and with a power which \emph{inverts} a hyperplane (i.e. which stabilises it and swaps the two halfspaces it delimits).
\end{itemize}
In order to illustrate and motivate our next definitions, let us consider an example. 

Let $X$ be a bi-infinite chain of squares such that any two consecutive squares intersect along a single vertex; see Figure \ref{PL}. We are interested in the cube complex $X \times \mathbb{R}$. Let $r$ denote the reflection of $X$ along the straight line passing through all the cut vertices; $t_1$ the natural translation of $X$; and $t_2$ the translation of $\mathbb{R}$. The hypothetic minimising set $M$ of $r \times t_2$ we want to define should contain an axis of $r \times t_2$; it should be invariant under the centraliser of $r \times t_2$; and it should split as a product such that $r \times t_2$ acts trivially on one factor. The latter condition, which holds in the CAT(0) setting, is fundamental in the applications to solvable groups. The axes of $r \times t_2$ are the lines $s \times \mathbb{R}$ where $s \in S := \{\text{cut-vertices of $X$}\}$. Therefore, $M$ must contain $s \times \mathbb{R}$ for some $s \in S$, and because it must be $\langle t_1,r \rangle$-invariant, there are only two possibilities: either $M= X \times \mathbb{R}$ or $M= S \times \mathbb{R}$. The former possibility seems reasonnable at first glance, but $r \times t_2$ does not act trivially on the factor $X$. Consequently, the only possibility is to define $M$ as $S \times \mathbb{R}$. 

This example shows that we cannot expect our new minimising sets to be connected subcomplexes. Nevertheless, they will be nicely embedded: they will be invariant under the \emph{median operator}, a ternary operator defined on the vertices of a CAT(0) cube complex which associates to each triple of vertices the unique vertex, referred to as the medien point, which belongs to a geodesic between any two of these vertices. See Section~\ref{section:Median} for more details. The key idea is to look for median sets instead convex or even isometrically embedded subcomplexes.

\begin{definition}\label{def:MedianSetsIntro}
Let $X$ be a CAT(0) cube complex and $g \in \mathrm{Isom}(X)$ an isometry. The \emph{median set} of $g$, denoted by $\mathrm{Med}(g)$, is
\begin{itemize}
	\item the union of all the $d$-dimensional cubes stabilised by $\langle g \rangle$ if $g$ elliptic, where $d$ denotes the minimal dimension of a cube stabilised by $\langle g \rangle$;
	\item the union of all the axes of $g$ if $g$ is loxodromic;
	\item the pre-image under $\pi$ of the union of all the axes of $g$ in $X/ \mathcal{J}$ if $g$ is inverting, where $\mathcal{J}$ is the collection of the hyperplanes inverted by powers of $g$ and where $\pi : X \to X/ \mathcal{J}$ is the canonical map to the cubical quotient $X/ \mathcal{J}$.
\end{itemize}
\end{definition}

\noindent
We refer to Section \ref{section:WallQuotient} for more information on cubical quotients. 

The main result of this article is that median sets are the good analogues of minimising sets of semisimple isometries in CAT(0) spaces. More precisely:

\begin{thm}\label{thm:MedianIntro}
Let $G$ be a group acting on a CAT(0) cube complex $X$ and $g \in G$ an isometry such that $\langle g \rangle$ is a normal subgroup of $G$. Then $\mathrm{Med}(g)$ is a median subalgebra of $X$ which is $G$-invariant and which decomposes as a product $T \times F \times Q$ of three median algebras $T,F,Q$ such that:
\begin{itemize}
	\item the action $G \curvearrowright T \times F \times Q$ decomposes as a product of three actions $G \curvearrowright T,F,Q$;
	\item $F$ is a single vertex if $g$ is elliptic, and otherwise it is a median flat on which $g$ acts by translations of length $\lim\limits_{n \to + \infty} \frac{1}{n} d(x,g^nx)>0$;
	\item $Q$ is a finite-dimensional cube, possibly reduced to a single vertex;
	\item $g$ acts trivially on $T$.
\end{itemize}
Moreover, the dimension of $Q$ is zero if $g$ is loxodromic, it coincides with the minimal dimension of a cube of $X$ stabilised by $g$ if $g$ is elliptic, and otherwise it coincides with the number of hyperplanes inverted by powers of $g$.
\end{thm}

We refer to Section \ref{section:PseudoLines} for the definition of \emph{median flats}. The idea to keep in mind is that a median flat is isomorphic to an isometrically embedded subcomplex of a Euclidean space.

\paragraph{Applications towards Question \ref{question:intro}.} Thanks to Theorem \ref{thm:MedianIntro}, we are able to prove the following cubical version of the famous flat torus theorem \cite[Theorem II.7.1]{MR1744486}:

\begin{thm}\label{intro:FTT}
Let $G$ be a group acting on a CAT(0) cube complex $X$ and $A \leq G$ a normal finitely generated abelian subgroup. Then there exist a finite-index subgroup $H \leq G$ containing $A$ and a median subalgebra $Y \subset X$ which is $H$-invariant and which decomposes as a product $T \times F \times Q$ of median algebras $T,F,Q$ such that:
\begin{itemize}
	\item $H \curvearrowright Y$ decomposes as a product of actions $H \curvearrowright T,F,Q$;
	\item $F$ is a median flat or a single point; 
	\item $Q$ is a finite-dimensional cube;
	\item $A$ acts trivially on $T$. 
\end{itemize}
\end{thm}

The main interest of our cubical flat torus theorem is that it provides an action on a median flat. Finding such an action is interesting because the isometry group of (the cubulation of) a median flat is quite specific, imposing severe restrictions on the action we are looking at. Our main application towards Question \ref{question:intro} is the following statement, stating that the action of a polycyclic group on a CAT(0) cube complex essentially factors through an abelian group:

\begin{thm}\label{intro:solvableInf}
Let $G$ be a polycyclic group acting on a CAT(0) cube complex $X$. Then 
\begin{itemize}
	\item $G$ contains a finite-index $H$ such that $$\mathcal{E}= \{ g \in H \mid \text{$g$ is $X$-elliptic} \}$$ defines a normal subgroup of $H$ and such that $H/\mathcal{E}$ is free abelian;
	\item and $H$ stabilises a median subalgebra which is either a median flat or a single point.
\end{itemize}
In particular, $G$ contains a finite-index subgroup which is (locally $X$-elliptic)-by-(free abelian).
\end{thm}

We emphasize that we are not imposing any restriction on the action and that the cube complex may be infinite-dimensional. As a consequence of Theorem \ref{intro:solvableInf}, it turns out that many solvable groups do not act properly on CAT(0) cube complexes:

\begin{cor}\label{intro:SolvableProper}
Let $G$ be a polycyclic group acting properly on a CAT(0) cube complex. Then $G$ must be virtually free abelian.
\end{cor}

The statement is essentially sharp, since many solvable groups which are not polycyclic act properly on CAT(0) cube complexes. For instance, for every $n \geq 1$ and every finite group $F$, the lamplighter group $F \wr \mathbb{Z}^n$ acts properly on a CAT(0) cube complex of dimension $2n$ \cite[Proposition 9.33]{Qm}; for every $n,m \geq 1$, the wreath product $\mathbb{Z}^m \wr \mathbb{Z}^n$ acts properly on an infinite-dimensional CAT(0) cube complex \cite{Haagerupwreath, MoiLamp}; and for every $n \geq 1$, the Houghton group $H_n$ acts properly on a CAT(0) cube complex of dimension $n$ \cite{LeeHoughton}, \cite[Example 4.3]{FarleyHughes}.

Corollary \ref{intro:SolvableProper} is also proved in \cite[Corollary 6.C.8]{CornulierCommensurated}, and another proof can be obtained by combining \cite{ConnerSolvable} and \cite{HaglundAxis}. However, we expect that Theorem \ref{intro:solvableInf} will be useful in the study of groups acting on (infinite-dimensional) CAT(0) cube complexes with infinite vertex-stabilisers. In a forthcoming article, we plan to apply this strategy in order to study polycyclic subgroups in braided Thompson-like groups.

\paragraph{Comparison to other cubical flat torus theorems.} Theorem \ref{intro:FTT} is not the first cubical version of the flat torus theorem which appears in the literature. The first one is \cite[Theorem 3.6]{WiseWood}, proving that, if a group acts geometrically on a CAT(0) cube complex, then its \emph{highest} virtually abelian subgroups act geometrically on convex subcomplexes isomorphic to products of quasi-lines. So this statement deals with geometric actions and only with specific virtually abelian subgroups, but the subcomplex which is constructed is convex. Arbitrary abelian subgroups are considered in \cite{Wood}, proving that, if a virtually abelian group acts properly on a CAT(0) cube complex, then it has to preserve an isometrically embedded subcomplex isomorphic to a product of quasi-lines. Compared to this statement, Theorem \ref{intro:FTT} does not make any assumption on the action, and it applies to groups containing normal abelian subgroups. The latter difference is the key of the article: it explains why we are able to study solvable groups instead of virtually abelian groups only. Another difference compared to \cite{WiseWood, Wood} is that we find median subsets instead of isometrically embedded subcomplexes. The reason is that, if we want to replace median subalgebras with isometrically embedded subcomplexes in the statement of Theorem \ref{intro:FTT}, then the conclusion may no longer hold, as illustrated by the example given before Definition \ref{def:MedianSetsIntro}. 
Nevertheless, it turns out that a well-chosen subdivision of (the cubulation of) a median subalgebra naturally embeds isometrically in the initial complex. As a consequence, the main theorem of \cite{Wood} can be recovered from Theorem \ref{intro:FTT}. However, we did not write the construction because it was not necessary to have subcomplexes for our applications.

\paragraph{Other applications.} Theorem \ref{thm:MedianIntro} has other interesting applications. First, we are able to prove that centralisers of infinite-order elements in cocompactly cubulated groups are themselves cocompactly cubulated, improving a result from \cite{HaettelArtin}:

\begin{thm}\label{info:centraliser}
Let $G$ be a group acting geometrically on a CAT(0) cube complex $X$. For every infinite-order element $g \in G$,  the centraliser $C_G(g)$ also acts geometrically on a CAT(0) cube complex.
\end{thm}

\noindent
Next, we are also able to deduce a splitting theorem (compare to \cite[Theorem~II.6.12]{MR1744486}), namely:

\begin{thm}\label{intro:splitting}
Let $G$ be a group acting on a CAT(0) cube complex $X$ and $A \lhd G$ a normal finitely generated subgroup. Assume that a non-trivial element of $A$ is never elliptic. Then $A$ is a direct factor of some finite-index subgroup of $G$.
\end{thm}

\noindent
Finally, by combining Theorem \ref{thm:MedianIntro} with the proof of \cite[Theorem B]{MR2665003}, we are able to deduce some valuable information about the possible actions of mapping class groups of surfaces (of succiently high complexity) on CAT(0) cube complexes.

\begin{thm}\label{thm:MCGintro}
Let $\Sigma$ be an orientable surface of finite type with genus $\geq 3$. Whenever $\mathrm{Mod}(\Sigma)$ acts on a CAT(0) cube complex, all Dehn twists are elliptic.
\end{thm}

\noindent
It is still an open problem to determine whether or not mapping class groups act fixed-point freely on CAT(0) cube complexes, and Theorem \ref{thm:MCGintro} might be a first step towards a solution. A positive answer would imply that such groups do not satisfy Kazhdan's property (T) according to \cite{MR1459140}, another open problem.

\paragraph{Organisation of the article.} Section \ref{section:pre} is dedicated to the definitions, basic properties and preliminary lemmas which will be used in the rest of the article. Next, in Section~\ref{section:PseudoLines}, we show that isometry groups of (cubulations of) median flats are quite specific and we show how to associate a median flat to any loxodromic isometry. In Section~\ref{section:MedianSets}, we prove Theorem \ref{thm:MedianIntro}, the main result of the article. Finally, all the applications of Theorem \ref{thm:MedianIntro} mentioned above are proved in Section \ref{section:applications}.

\paragraph{Acknowledgments.} I am grateful to Thomas Delzant for his interesting comments on a previous version of this article, and to the anonymous referee for their comments. This work was supported by a public grant as part of the Fondation Math\'ematique Jacques Hadamard.

\section{Preliminaries}\label{section:pre}

\subsection{Cube complexes, hyperplanes, projections}

\noindent
A \textit{cube complex} is a CW complex constructed by gluing together cubes of arbitrary (finite) dimension by isometries along their faces. It is \textit{nonpositively curved} if the link of any of its vertices is a simplicial \textit{flag} complex (i.e. $n+1$ vertices span a $n$-simplex if and only if they are pairwise adjacent), and \textit{CAT(0)} if it is nonpositively curved and simply-connected. See \cite[page 111]{MR1744486} for more information.

\medskip \noindent
Fundamental tools when studying CAT(0) cube complexes are \emph{hyperplanes}. Formally, a \textit{hyperplane} $J$ is an equivalence class of edges with respect to the transitive closure of the relation identifying two parallel edges of a square. Notice that a hyperplane is uniquely determined by one of its edges, so if $e \in J$ we say that $J$ is the \textit{hyperplane dual to $e$}. Geometrically, a hyperplane $J$ is rather thought of as the union of the \textit{midcubes} transverse to the edges belonging to $J$ (sometimes referred to as its \emph{geometric realisation}). See Figure \ref{figure27}. The \textit{carrier} $N(J)$ of a hyperplane $J$ is the union of the cubes intersecting (the geometric realisation of) $J$. 
\begin{figure}
\begin{center}
\includegraphics[trim={0 13cm 10cm 0},clip,scale=0.45]{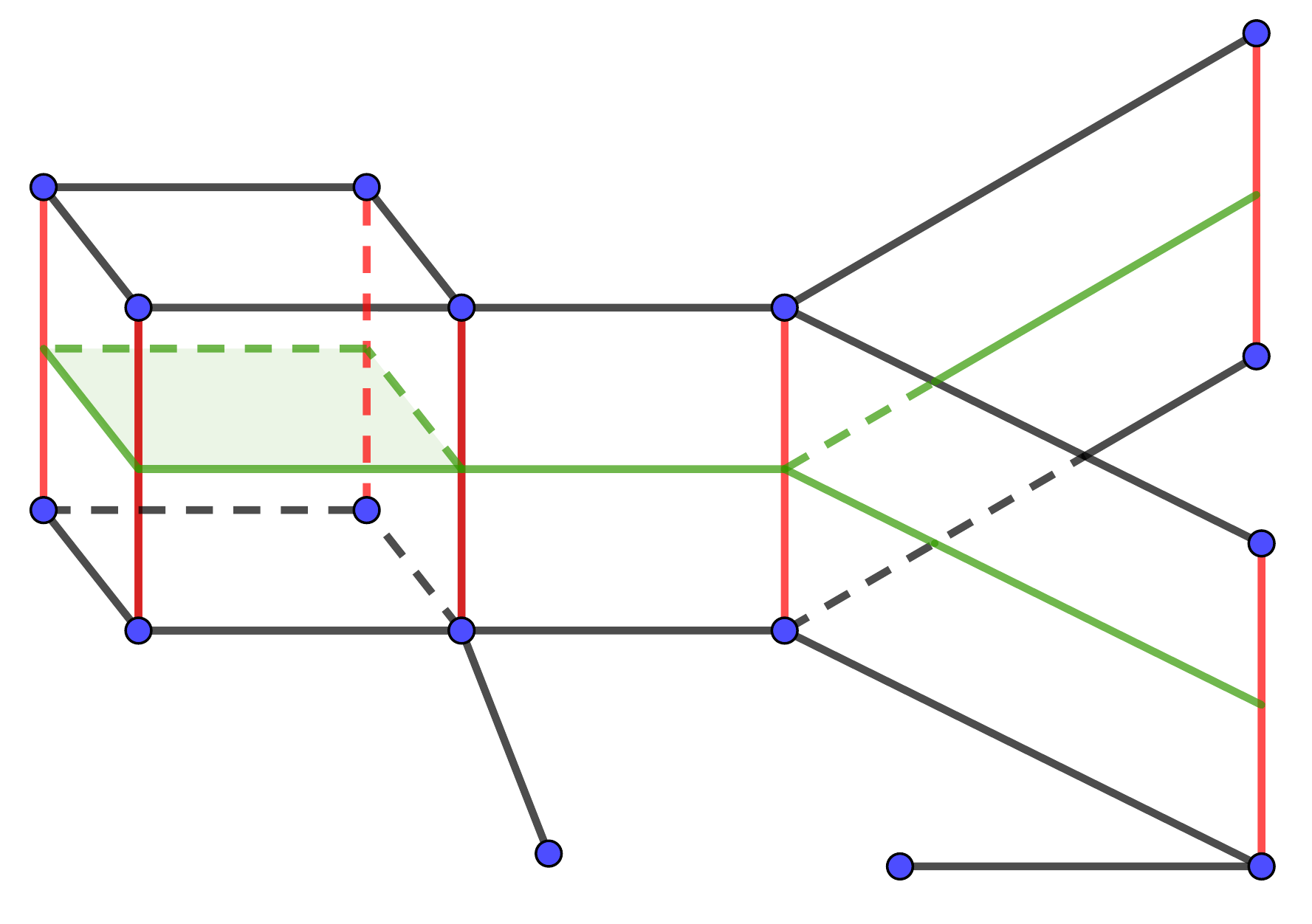}
\caption{A hyperplane (in red) and the associated union of midcubes (in green).}
\label{figure27}
\end{center}
\end{figure}

\medskip \noindent
There exist several metrics naturally defined on a CAT(0) cube complex. In this article, we are only interested in the graph metric defined on its one-skeleton, referred to as its \emph{combinatorial metric}. In fact, from now on, we will identify a CAT(0) cube complex with its one-skeleton, thought of as a collection of vertices endowed with a relation of adjacency. In particular, when writing $x \in X$, we always mean that $x$ is a vertex of $X$. 

\medskip \noindent
The following theorem will be often used along the article without mentioning it.

\begin{thm}\emph{\cite{MR1347406}}
Let $X$ be a CAT(0) cube complex.
\begin{itemize}
	\item If $J$ is a hyperplane of $X$, the graph $X \backslash \backslash J$ obtained from $X$ by removing the (interiors of the) edges of $J$ contains two connected components. They are convex subgraphs of $X$, referred to as the \emph{halfspaces} delimited by $J$.
	\item A path in $X$ is a geodesic if and only if it crosses each hyperplane at most once.
	\item For every $x,y \in X$, the distance between $x$ and $y$ coincides with the cardinality of the set $\mathcal{W}(x,y)$ of the hyperplanes separating them.
\end{itemize}
\end{thm}

\medskip \noindent
Another useful tool when studying CAT(0) cube complexes is the notion of \emph{projection} onto on a convex subcomplex, which is defined by the following proposition (see \cite[Lemma 13.8]{MR2377497}):

\begin{prop}\label{projection}
Let $X$ be a CAT(0) cube complex, $C \subset X$ a convex subcomplex and $x \in X \backslash C$ a vertex. There exists a unique vertex $y \in C$ minimizing the distance to $x$. Moreover, for any vertex of $C$, there exists a geodesic from it to $x$ passing through $y$.
\end{prop}

\noindent
Below, we record a couple of statements related to projections for future use. Proofs can be found in \cite[Lemma 13.8]{MR2377497} and \cite[Proposition 2.7]{article3} respectively.

\begin{lemma}\label{lem:HypProj}
Let $X$ be a CAT(0) cube complex, $Y \subset X$ a convex subcomplex and $x \in X$ a vertex. Any hyperplane separating $x$ from its projection onto $Y$ separates $x$ from $Y$. 
\end{lemma}

\begin{lemma}\label{lem:HypProjSeparate}
Let $X$ be a CAT(0) cube complex and $Y \subset X$ a convex subcomplex. For every vertices $x,y \in X$, the hyperplanes separating the projections of $x$ and $y$ onto $Y$ are precisely the hyperplanes separating $x$ and $y$ which cross $Y$. As a consequence the projection onto $Y$ is $1$-Lipschitz.
\end{lemma}

\subsection{Classification of isometries}\label{section:ClassIsom}

\noindent
Following \cite{HaglundAxis}, isometries of CAT(0) cube complexes can be classified into three families. Namely, given a CAT(0) cube complex $X$ and an isometry $g \in \mathrm{Isom}(X)$, 
\begin{itemize}
	\item $g$ is \emph{elliptic} if its orbits are bounded;
	\item $g$ is \emph{inverting} if its orbits are unbounded and if one of its powers \emph{inverts} a hyperplane, i.e. it stabilises it and swap the halfspaces it delimits;
	\item $g$ is \emph{loxodromic} if there exists a bi-infinite geodesic on which $g$ acts by translation of length $\|g\| := \min \{ d(x,gx) \mid x \in X\}$, referred to as an \emph{axis} of $g$. 
\end{itemize}
More precisely, Haglund showed that, if $g$ is neither elliptic nor inverting, then, for every vertex $z$ of the \emph{minimising set} 
$$\mathrm{Min}(g):= \left\{ x \in X \mid d(x,gx)= \min\limits_{y \in X} d(y,gy)\right\}$$
and for every geodesic $[z,gz]$ between $z$ and $gz$, the concatenation $\bigcup\limits_{k \in \mathbb{Z}} g^k \cdot [z,gz]$ is a geodesic on which $g$ acts by translations of length $\|g\|$. 

\medskip \noindent
We emphasize that, as shown by \cite[Proposition 5.2]{HaglundAxis}, if $g$ stabilises a geodesic and acts on it by translations of positive lengths, then it is necessarily an axis of $g$, that is to say, $g$ automatically acts on this geodesic by translations of length $\|g\|$. 

\medskip \noindent
Finally, let us mention that an isometry is elliptic if and only if it stabilises a (finite-dimensional) cube (see for instance \cite[Theorem 11.7]{Roller} for a proof). As a consequence, an alternative characterisation is that an isometry is elliptic if and only if one of its orbits is finite.

\begin{remark}
Given a collection of isometries, up to subdividing our cube complex we can always assume that isometries with bounded orbits fix vertices and that isometries with unbounded orbits are loxodromic. Consequently, it may be tempting to restrict ourself to this situation, and this would be sufficient to recover several of the applications mentioned in Section \ref{section:applications}. However, it is worth noticing that, if $g$ is a loxodromic isometry which stabilises a median subalgebra $Y$ (see Section \ref{section:Median} for a definition), then the isometry of $Y$ (when thought of as a CAT(0) cube complex on its own) induced by $g$ may be inverting. (For instance, set $X= \mathbb{R}^2$ and $Y= \mathbb{R} \times \{1\} \cup \mathbb{R} \times \{-1\}$, and let $g$ be the product of the translation $(1,0)$ with the reflection along $\mathbb{R} \times \{0\}$.) Therefore, considering inverting isometries and elliptic isometries without fixed-vertices is fundamental in the arguments used in Sections \ref{section:FTT} and \ref{section:Polycyclic}. Another reason to consider all the possible isometries is that, in explicit constructions of CAT(0) cube complexes, it may be more natural to work with the cube complex itself rather than its subdivision.
\end{remark}

\subsection{Wallspaces, cubical quotients}\label{section:WallQuotient}

\noindent
Given a set $X$, a \emph{wall} $\{A,B\}$ is a partition of $X$ into two non-empty subsets $A,B$, referred to as \emph{halfspaces}. Two points of $X$ are \emph{separated} by a wall if they belong to two distinct subsets of the partition. 

\medskip \noindent
A \emph{wallspace} $(X, \mathcal{W})$ is the data of a set $X$ and a collection of walls $\mathcal{W}$ such that any two points are separated by only finitely many walls. Such a space is naturally endowed with the pseudo-metric
$$d : (x,y) \mapsto \text{number of walls separating $x$ and $y$}.$$
As shown in \cite{ChatterjiNiblo, NicaCubulation}, there is a natural CAT(0) cube complex associated to any wallspace. More precisely, given a wallspace $(X, \mathcal{W})$, define an \emph{orientation} $\sigma$ as a collection of halfspaces such that:
\begin{itemize}
	\item for every $\{A,B\} \in \mathcal{W}$, $\sigma$ contains exactly one subset among $\{A,B\}$;
	\item if $A$ and $B$ are two halfspaces satisfying $A \subset B$, then $A \in \sigma$ implies $B \in \sigma$.
\end{itemize}
Roughly speaking, an orientation is a coherent choice of a halfspace in each wall. As an example, if $x \in X$, then the set of halfspaces containing $x$ defines an orientation. Such an orientation is referred to as a \emph{principal orientation}. Notice that, because any two points of $X$ are separated by only finitely many walls, two principal orientations are always \emph{commensurable}, i.e. their symmetric difference is finite.
\begin{figure}
\begin{center}
\includegraphics[trim={0 18cm 27cm 0},clip,scale=0.45]{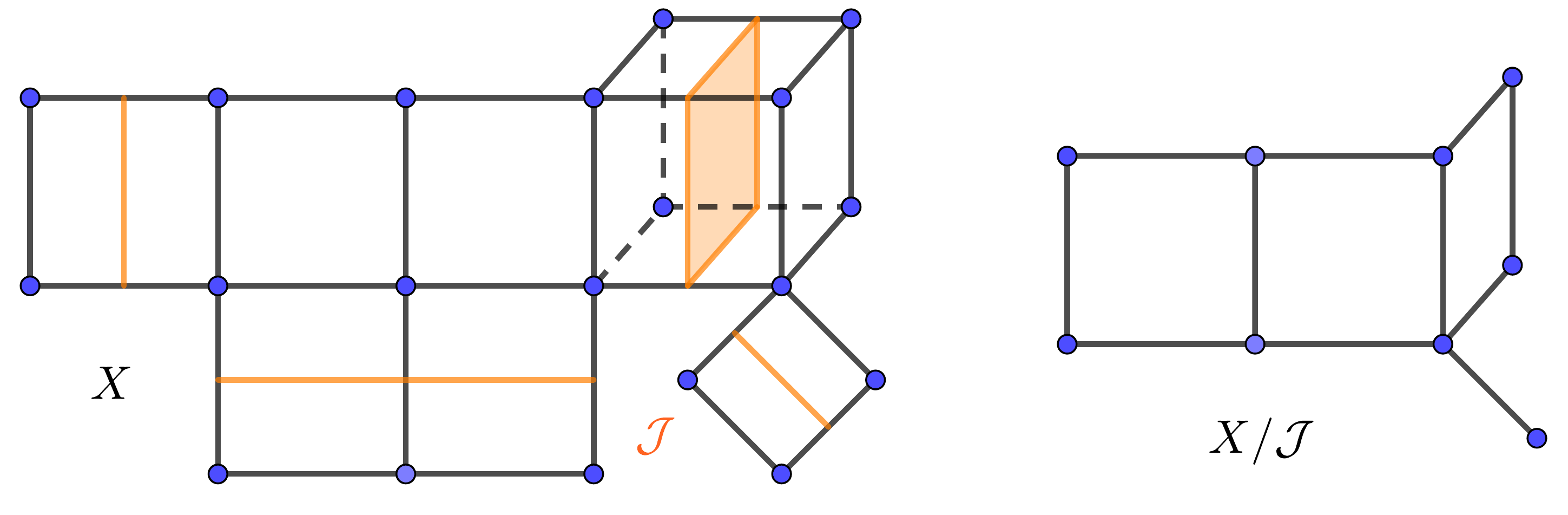}
\caption{A CAT(0) cube complex and one of its cubical quotients.}
\label{quotient}
\end{center}
\end{figure}

\medskip \noindent
The \emph{cubulation} of $(X, \mathcal{W})$ is the cube complex
\begin{itemize}
	\item whose vertices are the orientations within the commensurability class of principal orientations;
	\item whose edges link two orientations if their symmetric difference has cardinality two;
	\item whose $n$-cubes fill in all the subgraphs isomorphic to one-skeleta of $n$-cubes.
\end{itemize}
See Figure \ref{quotient} for an example.

\begin{definition}
Let $X$ be a CAT(0) cube complex and $\mathcal{J}$ be a collection of hyperplanes. Let $\mathcal{W}(\mathcal{J})$ denote the set of partitions of $X$ induced by the hyperplanes of $\mathcal{J}$. The \emph{cubical quotient} $X/ \mathcal{J}$ of $X$ by $\mathcal{J}$ is the cubulation of the wallspace $(X, \mathcal{W}(\mathcal{J})^c)$. 
\end{definition}

\noindent
See Figure \ref{quotient} for an example. It can be shown that $X/ \mathcal{J}$ can be obtained from $X$ in the following way. Given a hyperplane $J \in \mathcal{J}$, cut $X$ along $J$ to obtain $X \backslash \backslash J$. Each component of $X \backslash \backslash J$ contains a component of $N(J) \backslash \backslash J= N_1 \sqcup N_2$. Notice that $N_1$ and $N_2$ are naturally isometric: associate to each vertex of $N_1$ the vertex of $N_2$ which is adjacent to it in $N(J)$. Now, glue the two components of $X \backslash \backslash J$ together by identifying $N_1$ and $N_2$. The cube complex obtained is still CAT(0) and its set of hyperplanes corresponds naturally to $\mathcal{H}(X) \backslash \{J\}$ if $\mathcal{H}(X)$ denotes the set of hyperplanes of $X$. Thus, the same construction can repeated with a hyperplane of $\mathcal{J} \backslash \{J\}$, and so on. The cube complex which is finally obtained is the cubical quotient $X/ \mathcal{J}$.

\medskip \noindent
Notice that a quotient map is naturally associated to a cubical quotient, namely:
$$\left\{ \begin{array}{ccc} X & \to & X/ \mathcal{J} \\ x & \mapsto & \text{principal orientation defined by $x$} \end{array} \right..$$
The next lemma relates the distance between two vertices of $X$ to the distance between their images in $X/ \mathcal{J}$.

\begin{lemma}\label{lem:DistQuotient}
Let $X$ be a CAT(0) cube complex and $\mathcal{J}$ a collection of hyperplanes. If $\pi : X \to X/ \mathcal{J}$ denotes the canonical map, then
$$d_{X/ \mathcal{J}} \left( \pi(x), \pi(y) \right) = \# \left( \mathcal{W}(x,y) \backslash \mathcal{J} \right)$$
for every vertices $x,y \in X$. \qed
\end{lemma}

\noindent
Recall that, for every vertices $x,y \in X$, the set $\mathcal{W}(x,y)$ denotes the collection of the hyperplanes of $X$ separating $x$ and $y$.

\subsection{Roller boundary}

\noindent
Let $X$ be a CAT(0) cube complex. An \emph{orientation} of $X$ is an orientation of the wallspace $(X, \mathcal{W}(\mathcal{J}))$, as defined in the previous section, where $\mathcal{J}$ is the set of all the hyperplanes of $X$. The \emph{Roller compactification} $\overline{X}$ of $X$ is the set of the orientations of $X$. Usually, we identify $X$ with the image of the embedding
$$\left\{ \begin{array}{ccc} X & \to & \overline{X} \\ x & \mapsto & \text{principal orientation defined by $x$} \end{array} \right.$$
and we define the \emph{Roller boundary} of $X$ by $\mathfrak{R}X:= \overline{X} \backslash X$. 

\medskip \noindent
The Roller compactification is naturally endowed with a topology via the inclusion $\overline{X} \subset 2^{ \{ \text{halfspaces} \}}$ where $2^{ \{ \text{halfspaces} \}}$ is endowed with the product topology. Otherwise saying, a sequence of orientations $(\sigma_n)$ converges to $\sigma \in \overline{X}$ if, for every finite collection of halfspaces $\mathcal{D}$, there exists some $N \geq 0$ such that $\sigma_n \cap \mathcal{D}=\sigma \cap \mathcal{D}$ for every $n \geq N$. 

\medskip \noindent
The Roller compactification is also naturally a cube complex. Indeed, if we declare that two orientations are linked by an edge if their symmetric difference has cardinality two and if we declare that any subgraph isomorphic to the one-skeleton of an $n$-cube is filled in by an $n$-cube for every $n \geq 2$, then $\overline{X}$ is a disjoint union of CAT(0) cube complexes. Each such component is referred to as a \emph{cubical component} of $\overline{X}$. See Figure \ref{Roller} for an example. Notice that the distance (possibly infinite) between two vertices of $\overline{X}$ coincides with the number of hyperplanes which separate them, if we say that a hyperplane $J$ separates two orientations when they contain different halfspaces delimited by $J$. 
\begin{figure}
\begin{center}
\includegraphics[trim={0 12cm 41cm 0},clip,scale=0.45]{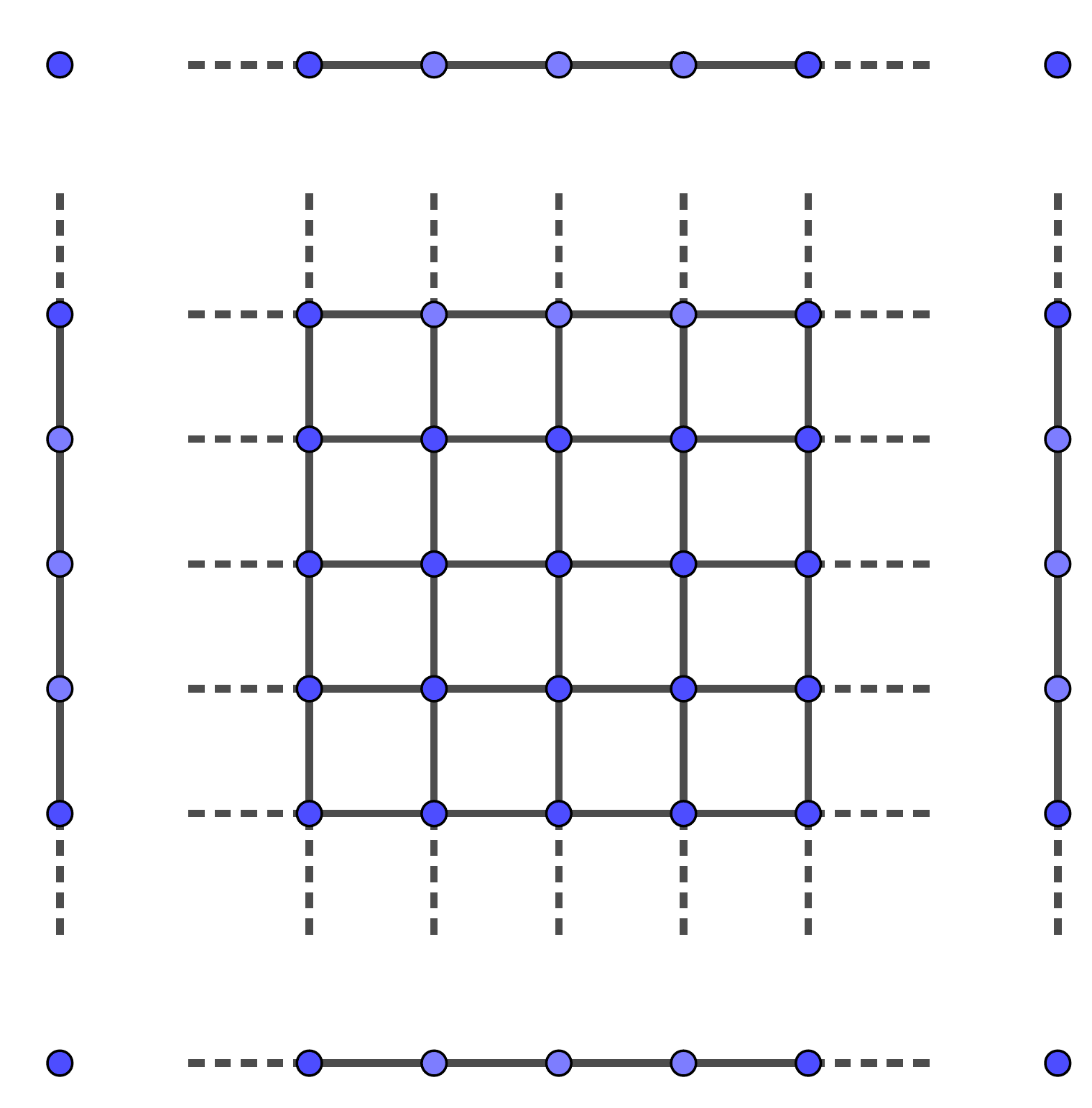}
\caption{Roller compactification of $\mathbb{R}^2$. It contains nine cubical components.}
\label{Roller}
\end{center}
\end{figure}

\medskip \noindent
An alternative description of the Roller boundary is the following. Let $X$ be a CAT(0) cube complex and $x \in X$ a basepoint. Denote by $\mathfrak{S}_xX$ the collection of the geodesic rays of $X$ starting from $x$ up to the equivalence relation which identifies two rays if they cross exactly the same hyperplanes. According to \cite[Proposition A.2]{article3}, the map
$$\left\{ \begin{array}{ccc} \mathfrak{S}_xX & \to & \mathfrak{R}X \\ r & \mapsto & \alpha(r) \end{array} \right.$$
is a bijection, where $\alpha(r)$ denotes the orientation containing the halfspaces in which the ray $r$ is eventually included. Notice that the metric in the cube complex $\overline{X}$ corresponds to the metric $(r_1,r_2) \mapsto \# \left( \mathcal{H}(r_1) \Delta \mathcal{H}(r_2) \right)$ in $\mathfrak{S}_xX$, where $\mathcal{H}(\cdot)$ denotes the collection of the hyperplanes which are crossed by the ray we are looking at and where $\Delta$ denotes the symmetric difference. 

\medskip \noindent
We conclude this subsection by proving a lemma which will be useful later.

\begin{lemma}\label{lem:SubEuclidean}
Let $X$ be a CAT(0) cube complex which is isomorphic to an isometrically embedded subcomplex of $\mathbb{R}^n$ for some $n \geq 1$. Then $\overline{X}$ has finitely many cubical components and one of them is bounded.
\end{lemma}

\begin{proof}
If $X$ is bounded, there is nothing to prove, so we suppose that $X$ is unbounded. Without loss of generality, suppose that $X$ is an isometrically embedded subcomplex of $\mathbb{R}^n$, where $n \geq 1$. As a consequence of the description of $\mathfrak{R}X$ in terms of geodesic rays, if we choose $(0, \ldots, 0)$ as our basepoint then any point of $\mathfrak{R}X$ can be represented by an element of $\overline{\mathbb{Z}}^n$ where $\overline{\mathbb{Z}}= \mathbb{Z} \cup \{\pm \infty\}$. Moreover, the distance (possibly infinite) in the cube complex $\mathfrak{R}X$ (which coincides with the distance in $\mathfrak{R}\mathbb{R}^n$ as $X$ is isometrically embedded) between two points represented by $(a_1, \ldots, a_n)$ and $(b_1, \ldots, b_n)$ coincides with 
$$\sum\limits_{i=1}^n |a_i-b_i|, \ \text{where by convention} \ \left\{ \begin{array}{l} \infty- \infty= - \infty+ \infty=0 \\ \infty+ \infty=\infty \\ - \infty - \infty=- \infty \end{array} \right..$$
As a consequence, two distinct points of $\mathfrak{R}X$ cannot be represented by the same element of $\overline{\mathbb{Z}}^n$. From this description, it is clear that $\mathfrak{R}X$ has at most $3^n-1$ cubical components, proving the first assertion of our lemma.

\medskip \noindent
Now fix a point of $\alpha \in \mathfrak{R}X$ whose representation in $\overline{a} \in \overline{\mathbb{Z}}^n$ contains a maximal number of infinite coordinates. Up to permuting and inverting the coordinates of our representation, we may suppose without loss of generality that $\overline{a}=(a_1, \ldots, a_r, + \infty, \ldots,+ \infty)$ where $r \geq 0$ and $a_1, \ldots, a_r \in \mathbb{Z}$. Suppose by contradiction that the cubical component of $\mathfrak{R}X$ containing $\alpha$ is unbounded. It implies that, once again up to permuting and inverting coordinates, there exist some $s \geq 1$, some $b_1, \ldots, b_{r-s-1} \in \mathbb{Z}$ and increasing sequences $(a_{r-s}(k)), \ldots, (a_r(k))$ such that 
$$(b_1, \ldots, b_{r-s-1}, a_{r-s}(k), \ldots, a_r(k),+ \infty, \ldots, + \infty)$$
represents a point of $\mathfrak{R}X$ for every $k \geq 0$. By construction of our representation, there exist increasing sequences $(c_{n-r}(p)), \ldots, (c_n(p))$ such that
$$(b_1, \ldots, b_{r-s-1}, a_{r-s}(k), \ldots, a_r(k), c_{n-r}(p), \ldots, c_n(p))$$
represents a point of $X$ for every $k \geq 0$ and $p \geq 0$. Set
$$q(k) : = (b_1, \ldots, b_{r-s-1}, a_{r-s}(k), \ldots, a_r(k), c_{n-r}(k), \ldots, c_n(k))$$
for every $k \geq 0$. Notice that the fact that our sequences are increasing implies that the points $q(0),q(1), \ldots$ all belong to a common geodesic ray starting from $(0, \ldots, 0)$. As a consequence, the point $(b_1, \ldots, b_{r-s-1}, + \infty, \ldots, + \infty)$ represents a point of $\mathfrak{R}X$, contradicting the definition of $\alpha$.

\medskip \noindent
Thus, we have proved that the cubical component of $\mathfrak{R}X$ containing $\alpha$ is bounded, concluding the proof of our lemma. 
\end{proof}

\subsection{Median algebras}\label{section:Median}

\noindent
A \emph{median algebra} $(X,\mu)$ is the data of a set $X$ and a map $\mu : X \times X \times X \to X$ satisfying the following conditions:
\begin{itemize}
	\item $\mu(x,y,y)=y$ for every $x,y \in X$;
	\item $\mu(x,y,z)= \mu(z,x,y)= \mu (x,z,y)$ for every $x,y,z \in X$;
	\item $\mu\left( \mu(x,w,y), w,z \right) = \mu \left( x,w, \mu(y,w,z) \right)$ for every $x,y,z,w \in X$.
\end{itemize}
The \emph{interval} between two points $x,y \in X$ is
$$I(x,y)= \left\{ z \in X \mid \mu(x,y,z)=z \right\};$$
and a subset $Y \subset X$ is \emph{convex} if $I(x,y) \subset Y$ for every $x,y \in Y$. In this article, we are only interested in median algebras whose interval are finite; they are referred to as \emph{discrete} median algebras.

\medskip \noindent
As proved in \cite{NicaCubulation}, a discrete median algebra is naturally a wallspace. Indeed, let us say that $Y \subset X$ is a \emph{halfspace} if $Y$ and $Y^c$ are both convex. Then a \emph{wall} of $X$ is the data of halfspace and its complement, and it turns out that only finitely many walls separate two given point of $X$. The \emph{cubulation} of a discrete median algebra refers to the cubulation of this wallspace. In this specific case, it turns out that any orientation commensurable to a principal orientation must be a principal orientation itself. Consequently, the cubulation of a discrete median algebra $X$ coincides with the cube complex
\begin{itemize}
	\item whose vertex-set is $X$;
	\item whose edges link two points of $X$ if they are separated by a single wall;
	\item whose $n$-cubes fill in every subgraph of the one-skeleton isomorphic to the one-skeleton of an $n$-cube, for every $n \geq 2$. 
\end{itemize}
Therefore, a discrete median algebra may be naturally identified with its cubulation, and so may be thought of as a CAT(0) cube complex. The \emph{dimension} and the \emph{Roller compactification} of a discrete median algebra coincides with the dimension and the Roller compactification of its cubulation. 

\medskip \noindent
Conversely, a CAT(0) cube complex $X$ naturally defines a discrete median algebra (see \cite{mediangraphs} and \cite[Proposition 2.21]{MR2413337}). Indeed, for every triple of vertices $x,y,z \in X$, there exists a unique vertex $\mu(x,y,z) \in X$ satisfying
$$\left\{ \begin{array}{l} d(x,y)= d( x , \mu(x,y,z))+ d(\mu(x,y,z), y) \\ d(x,z)= d(x,\mu(x,y,z))+d(\mu(x,y,z),z) \\ d(y,z)= d(y,\mu(x,y,z))+d(\mu(x,y,z),z) \end{array} \right..$$
Otherwise saying, $I(x,y) \cap I(y,z) \cap I(x,z) = \{ \mu(x,y,z)\}$. The vertex $\mu(x,y,z)$ is referred to as the \emph{median point} of $x,y,z$. Then $(X,\mu)$ is a discrete median algebra, motivating the following terminology:

\begin{definition}
Let $X$ be a CAT(0) cube complex. A \emph{median subalgebra} $Y \subset X$ is a set of vertices stable under the median operation.
\end{definition}

\noindent
The median structure defined on $X$ can be extended continuously to the Roller compactification $\overline{X}$. Namely, given three orientations $\sigma_1,\sigma_2,\sigma_3$, define $\mu(\sigma_1,\sigma_2,\sigma_3)$ as the set of halfspaces which belong to at least two orientations among $\{\sigma_1,\sigma_2,\sigma_3\}$. Then the map $\mu : \overline{X} \times \overline{X} \times \overline{X} \to \overline{X}$ extends the previous median operation, making $(\overline{X},\mu)$ a median algebra (not discrete in general). Moreover, the ternary operator turns out to be continuous when $\overline{X}$ is endowed with the topology defined in the previous section. 

\medskip \noindent
We conclude this subsection by stating and proving a few lemmas which will be useful later. 

\begin{lemma}
Let $(X_1,\mu_1)$ and $(X_2, \mu_2)$ be two discrete median algebras of dimensions $d_1$ and $d_2$ respectively. Then $(X_1 \times X_2, \mu_1 \times \mu_2)$ is a discrete median algebra of dimension $d_1+d_2$.
\end{lemma}

\noindent
The proof is straightforward and it is left to the reader.

\begin{lemma}\label{lem:TwoCubulations}
Let $X$ be a CAT(0) cube complex and $Y \subset X$ a median subalgebra. The cubulation of $Y$ coincides with the cubulation of the wallspace $(Y, \mathcal{W})$ where $\mathcal{W}$ denotes the set of partitions of $Y$ induced by the hyperplanes of $X$ (without multiplicity).
\end{lemma}

\begin{proof}
It is sufficient to show that, for every $D \subset Y$, $D$ is a halfspace of the median algebra $Y$ if and only if $D$ is the \emph{trace} of a halfspace of $X$ (i.e. there exists a halfspace $H$ of $X$ such that $D= H \cap Y$). It is clear that the trace of a halfspace of $X$ provides a halfspace of the median algebra of $Y$, so let $D$ be a halfspace of the median algebra $Y$. 

\medskip \noindent
Fix two points $x \in D$ and $y \in D^c$ minimising the distance between $D$ and $D^c$. Fixing a hyperplane $J$ of $X$ separating $x$ and $y$, we claim that $J$ separates $D$ and $D^c$. For convenience, let $J^-,J^+$ denote the halfspaces delimited by $J$ such that $x \in J^-$ and $y \in J^+$. 

\medskip \noindent
Let $z \in Y$ be a point which is not separated from $x$ by $J$. Notice that the median point $\mu(x,y,z)$ has to belong to $J^-$ by convexity of $J^-$. In particular, $d(y, \mu(x,y,z))>0$. Because $Y$ is median, it has to contain $\mu(x,y,z)$. But, since $d(x, \mu(x,y,z))<d(x,y)$, $\mu(x,y,z)$ cannot belong to $D^c$, hence $\mu(x,y,z) \in D$. Consequently,
$$\begin{array}{lcl} d(x, \mu(x,y,z)) & = & d(x,y)-d(\mu(x,y,z),y) = d(y,D)- d(\mu(x,y,z),y) \\ \\ & \leq & d(y, \mu(x,y,z))- d(\mu(x,y,z),y)=0, \end{array}$$
hence $\mu(x,y,z)=x$. Therefore, $\mu(x,y,z)$ belongs to the interval between $y$ and $z$. As a consequence, it $z$ belongs to $D^c$, then $\mu(x,y,z)$ has to belong to $D^c$ by convexity, which is not the case, so $z \in D$. 

\medskip \noindent
Thus, we have proved that $J^- \cap Y \subset D$. We show similarly that $J^+ \cap Y \subset D^c$, hence $J^- \cap Y=D$ and $J^+ \cap Y=D^c$. The proof of our lemma is complete. 
\end{proof}

\begin{lemma}\label{lem:HypSepMedian}
Let $X$ be a CAT(0) cube complex, and let $\{a,b,c \}$ and $\{x,y,z\}$ be two triples of vertices of $X$. If a hyperplane $J$ separates the median points $\mu(a,b,c)$ and $\mu(x,y,z)$, then it has to separate $a$ and $x$, or $b$ and $y$, or $c$ and $z$.
\end{lemma}

\begin{proof}
Let $J^-$ and $J^+$ denote the halfspaces delimited by $J$ which contain $\mu(a,b,c)$ and $\mu(x,y,z)$ respectively. By convexity of halfspaces, $J^-$ must contain at least two vertices of $\{a,b,c\}$ and similarly $J^+$ must contain at least two vertices of $\{x,y,z\}$. If $J^-$ does not contain $a$, then $J^+$ has to contain either $y$ or $z$, so that $J$ separates $b$ and $y$ or $c$ and $z$. One argues similarly if $J^-$ does not contain $b$ or $c$, concluding the proof of our lemma.
\end{proof}

\subsection{Busemann morphisms}

\noindent
An interesting observation is that, if $T$ is a simplicial tree and $\alpha \in \partial T$ a point at infinity, then there exists a natural morphism $\beta_\alpha : \mathrm{stab}(\alpha) \to \mathbb{Z}$ which can be described in the following way. Fixing a basepoint $x \in T$ and an isometry $g \in \mathrm{stab}(\alpha)$, notice that $g \cdot [x,\alpha) \cap [x, \alpha)$ contains a infinite subray. So $g$ acts like a translation on a ray pointing to $\alpha$. The value of $\beta_\alpha(g)$ is precisely the length of this translation (positive it is directed to $\alpha$, and negative otherwise). Moreover, it follows from this description that the kernel of our morphism coincides with the set of elliptic isometries of $\mathrm{stab}(\alpha)$.

\medskip \noindent
The goal of this section is to explain how to extend such a construction to finite-dimensional CAT(0) cube complexes. More precisely, we want to prove the following (slight) improvement of \cite[Theorem B.1]{CFI}:

\begin{thm}\label{thm:Busemann}
Let $X$ be a finite-dimensional CAT(0) cube complex and $\alpha \in \mathfrak{R}X$ a point at infinity. There exist a subgroup $\mathrm{stab}_0(\alpha)$ of $\mathrm{stab}(\alpha) \leq \mathrm{Isom}(X)$ of index at most $\dim(X)!$ and a morphism $\beta : \mathrm{stab}_0(\alpha) \to \mathbb{Z}^n$, where $n \leq \dim(X)$, such that 
$$\mathrm{ker}(\beta)= \{ g \in \mathrm{stab}_0(\alpha) \mid \text{$g$ is $X$-elliptic} \}.$$
Consequently, $\mathrm{stab}(\alpha)$ is virtually (locally $X$-elliptic)-by-(free abelian).
\end{thm}

\noindent
The construction of such a morphism comes from the appendix of \cite{CFI}. First of all, we need the following statement \cite[Proposition 4.H.1]{CornulierCommensurated}:

\begin{lemma}
Let $G$ be a group acting on a set $S$. Assume that there exists a subset $M \subset S$ such that the symmetric difference $gM \Delta M$ is finite for every $g \in G$. Then
$$\mathrm{tr}_M : g \mapsto |M \backslash g^{-1}M | - |g^{-1}M \backslash M|$$
defines a morphism $G \to \mathbb{Z}$. Moreover, if $N \subset S$ is another subset such that $M \Delta N$ is finite, then $\mathrm{tr}_M= \mathrm{tr}_N$.
\end{lemma}

\noindent
Next, we need to introduce some vocabulary. A collection of hyperplanes $\mathcal{U}$ is:
\begin{itemize}
	\item \emph{inseparable} if each hyperplane separating two elements of $\mathcal{U}$ belongs to $\mathcal{U}$;
	\item \emph{unidirectional} if, for every $J \in \mathcal{U}$, one of the two halfspaces bounded by $J$ contains only finitely many elements of $\mathcal{U}$;	
	\item a \emph{facing triple} if it is a collection of three pairwise non-transverse hyperplanes none of them separating the other two;
	\item a \emph{UBS} (for \emph{Unidirectional Boundary Set}) if it is infinite, inseparable, unidirectional and contains no facing triple;
	\item a \emph{minimal UBS} if every UBS contained into $\mathcal{U}$ has finite symmetric difference with~$\mathcal{U}$;
	\item \emph{almost transverse} to another collection of hyperplanes $\mathcal{V}$ if each $J \in \mathcal{U}$ crosses all but finitely many hyperplanes in $\mathcal{V}$ and if each $H \in \mathcal{V}$ crosses all but finitely many hyperplanes in $\mathcal{U}$. 
\end{itemize}
The link between UBS and the Roller boundary is made explicit by the following result \cite[Lemma B.7]{CFI}:

\begin{lemma}
Let $X$ be a CAT(0) cube complex, $x \in X$ a basepoint and $\alpha \in \mathfrak{R}X$ a point at infinity. The collection $\mathcal{U}(x,\alpha)$ of the hyperplanes separating $x$ and $\alpha$ is a UBS.
\end{lemma}

\noindent
The only thing we need to know about UBS is the following decomposition result \cite[Theorem 3.10]{simplicialboundary}:

\begin{prop}\label{prop:Hagen}
Let $X$ be a finite-dimensional CAT(0) cube complex. Given a UBS $\mathcal{U}$, there exists a UBS $\mathcal{U}'$ which has finite symmetric difference with $\mathcal{U}$ and which is partioned into a union of $k \leq \dim(X)$ pairwise almost transverse minimal UBS, say $\mathcal{U}_1, \ldots, \mathcal{U}_k$. Moreover, if $\mathcal{U}''$ is another UBS which has finite symmetric difference with $\mathcal{U}$ and which is partitioned into a finite union of minimal UBS $\mathcal{U}_1'', \ldots, \mathcal{U}_r''$ which are pairwise almost transverse, then $k=r$ and, up to reordering, the symmetric difference between $\mathcal{U}_i$ and $\mathcal{U}_i''$ is finite for every $i$. 
\end{prop}

\noindent
We are now ready to extend Busemann morphisms from tree to finite-dimensional CAT(0) cube complexes. 

\begin{definition}
Let $X$ be a finite-dimensional CAT(0) cube complex and $\alpha \in \mathfrak{R}X$ a point at infinity. Fixing a basepoint $x \in X$, let $\mathcal{U}_1 \cup \cdots \cup \mathcal{U}_n$ be the decomposition of $\mathcal{U}(x,\alpha)$ as a union of minimal UBS. Denote by $\mathrm{stab}_0(\alpha)$ the finite-index subgroup of $\mathrm{stab}(\alpha)$ which preserves the commensurability classes of $\mathcal{U}_1, \ldots, \mathcal{U}_n$. The \emph{Busemann morphim} of $\alpha$ is
$$\beta : = \bigoplus\limits_{i=1}^n \mathrm{tr}_{\mathcal{U}_i} : \mathrm{stab}_0(\alpha) \to \mathbb{Z}^n.$$
It does not depend on the choice of the basepoint $x$.
\end{definition}

\noindent
Our theorem can be proved now.

\begin{proof}[Proof of Theorem \ref{thm:Busemann}.]
The inclusion $\mathrm{ker}(\beta) \subset \{ g \in \mathrm{stab}_0(\alpha) \mid \text{$g$ is $X$-elliptic} \}$ follows from the proof of \cite[Theorem B.1]{CFI}. Conversely, let $g \in \mathrm{stab}_0(\alpha)$ be an $X$-elliptic isometry. As mentioned in Section \ref{section:ClassIsom}, $g$ has to stabilise a cube. Up to replacing $g$ with one of its powers, we will suppose that $g$ fixes a vertex $x \in X$. Decompose $\mathcal{U}(x,\alpha)$ as the disjoint union $\mathcal{H}_0 \sqcup \mathcal{H}_1 \sqcup \cdots$ where $\mathcal{H}_i = \{ J \in \mathcal{U}(x,\alpha) \mid d(x,N(J))=i\}$ for every $i \geq 0$. Clearly, each $\mathcal{H}_i$ is a $\langle g \rangle$-invariant collection of pairwise transverse hyperplanes; it follows in particular that $\mathcal{H}_i$ has cardinality at most $\dim(X)$. As a consequence, $g^{\dim(X)!}$ stabilises each hyperplane of $\mathcal{U}(x,\alpha)$. Let $\mathcal{U}_1 \cup \cdots \cup \mathcal{U}_n$ denote the decomposition of a UBS $\mathcal{U}$ which has finite symmetric difference with $\mathcal{U}(x,\alpha)$ as given by Proposition \ref{prop:Hagen}. Up to replacing $\mathcal{U}$ with $\mathcal{U} \cap \mathcal{U}(x,\alpha)$, we may suppose without loss of generality that $\mathcal{U} \subset \mathcal{U}(x, \alpha)$. We deduce that
$$\beta(g)= \frac{1}{\dim(X)!} \cdot \beta \left( g^{\dim(X)!} \right) = \frac{1}{\dim(X)!} \left( \mathrm{tr}_{\mathcal{U}_1} \left( g^{\dim(X)!} \right), \ldots, \mathrm{tr}_{\mathcal{U}_n} \left( g^{\dim(X)!} \right) \right) = 0,$$
i.e. $g \in \mathrm{ker}(g)$. The proof of the theorem is complete. 
\end{proof}

\section{Median flats and their isometries}\label{section:PseudoLines}

\noindent
In the traditional flat torus theorem, the abelian subgroup acts on a product of geodesic lines. In our cubical version of the theorem, these subspaces are replaced with \emph{median flats}. This section is dedicated to the study of these median algebras.

\begin{definition}
A \emph{median flat} is a non-empty, finite-dimensional, discrete median algebra which is included into the interval of two points of its Roller boundary.
\end{definition}

\noindent
For instance, the Euclidean spaces $\mathbb{Z}^n$ are median flats, as well as the chain of squares illustrated by Figure \ref{PL}. However, $[0,1] \times \mathbb{Z}$ or a line on which we glue an additional edge are not median flats. It is worth noticing that being a median flat is preserved under products:

\begin{lemma}\label{lem:ProductFlats}
A product of finitely many median flats is a median flat.
\end{lemma}

\begin{proof}
Let $F_1, \ldots, F_n$ be median flats. So, for every $1 \leq i \leq n$, there exist $\zeta_i,\xi_i \in \mathfrak{R}F_i$ such that $F_i$ coincides with the interval $I(\zeta_i,\xi_i)$. For every $(a_1, \ldots, a_n) \in F_1 \times \cdots \times F_n$, we have
$$\begin{array}{lcl} \mu \left( (\zeta_1, \ldots, \zeta_n), (a_1, \ldots, a_n), (\xi_1, \ldots, \xi_n) \right) & = & \left( \mu(\zeta_1,a_1,\xi_1), \ldots, \mu( \zeta_n, a_n, \xi_n) \right) \\ \\ & = & \left( a_1, \ldots, a_n \right). \end{array}$$
Therefore, $F_1 \times \cdots \times F_n$ coincides with the interval between $(\zeta_1, \ldots, \zeta_n)$ and $(\xi_1, \ldots, \xi_n)$.
\end{proof}

\noindent
In order to motivate the analogy between products of geodesic lines in CAT(0) spaces and median flats, let us mention that, as a consequence of \cite[Theorem 1.14]{NibloPropA}, a flat always embeds into a Euclidean space:

\begin{lemma}\label{lem:FlatInEuclidean}
Let $F$ be a median flat of dimension $d$. Then its cubulation isometrically embeds into the cube complex $\mathbb{Z}^d$. 
\end{lemma}

\noindent
The main interest of our cubical flat torus theorem is that it provides an action on a median flat. Finding such an action is interesting because the isometry group of (the cubulation of) a median flat is quite specific, imposing severe restrictions on the action we are looking at. The following proposition motivates this idea:

\begin{prop}\label{prop:Busemann}
Let $X$ be the cubulation of a median flat. There exist a finite-index subgroup $\mathrm{Isom}_0(X) \leq \mathrm{Isom}(X)$, an integer $n \geq 1$ and a morphism $\beta : \mathrm{Isom}_0(X) \to \mathbb{Z}^n$ such that $\mathrm{ker}(\beta) = \{ g \in \mathrm{Isom}_0(X) \mid \text{$g$ is $X$-elliptic} \}$. In particular, $\mathrm{ker}(\beta)$ is locally $X$-elliptic. 
\end{prop}

\noindent
A subgroup is called \emph{locally elliptic} if every finitely generated subgroup is elliptic (or equivalently in CAT(0) cube complexes, has a finite orbit or stabilises a cube).

\begin{proof}[Proof of Proposition \ref{prop:Busemann}.]
Let $\mathrm{Isom}^+(X)$ the subgroup of $\mathrm{Isom}(X)$ containing the isometries that stabilise each cubical component of $\mathfrak{R}X$. Notice that, by combining Lemmas \ref{lem:FlatInEuclidean} and \ref{lem:SubEuclidean}, we know that $X$ has only finitely many cubical components, so that $\mathrm{Isom}^+(X)$ has to be a finite-index subgroup of $\mathrm{Isom}(X)$. We also know that $\mathfrak{R}X$ contains a cubical component which is bounded. Therefore, there exists some $\alpha \in \mathfrak{R}X$ such that $\mathrm{Isom}^+(X) \cap \mathrm{stab}(\alpha)$ has finite index in $\mathrm{Isom}(X)$. According to Theorem \ref{thm:Busemann}, there exists a finite-index subgroup $\mathrm{stab}_0(\alpha) \leq \mathrm{stab}_{\mathrm{Isom}(X)}(\alpha)$, an integer $n \geq 1$ and a morphism $\beta : \mathrm{stab}_0(\alpha) \to \mathbb{Z}^n$ such that $\mathrm{ker}(\varphi)= \{ g \in \mathrm{stab}_0(\alpha) \mid \text{$g$ is $X$-elliptic} \}$. Therefore, the restriction of $\beta$ to $\mathrm{Isom}_0(X):= \mathrm{Isom}^+(X) \cap \mathrm{stab}_0(\alpha)$ provides the desired morphism.
\end{proof}

\noindent
We emphasize that the isometry group of a median flat is not necessarily virtually free abelian. For instance, the median flat illustrated by Figure \ref{PL} has its isometry group isomorphic to the wreath product $\mathbb{Z}_2 \wr D_\infty$. Nevertheless, any group acting properly on a median flat turns out to be virtually free abelian.  
\begin{figure}
\begin{center}
\includegraphics[trim={0 23cm 35cm 0},clip,scale=0.45]{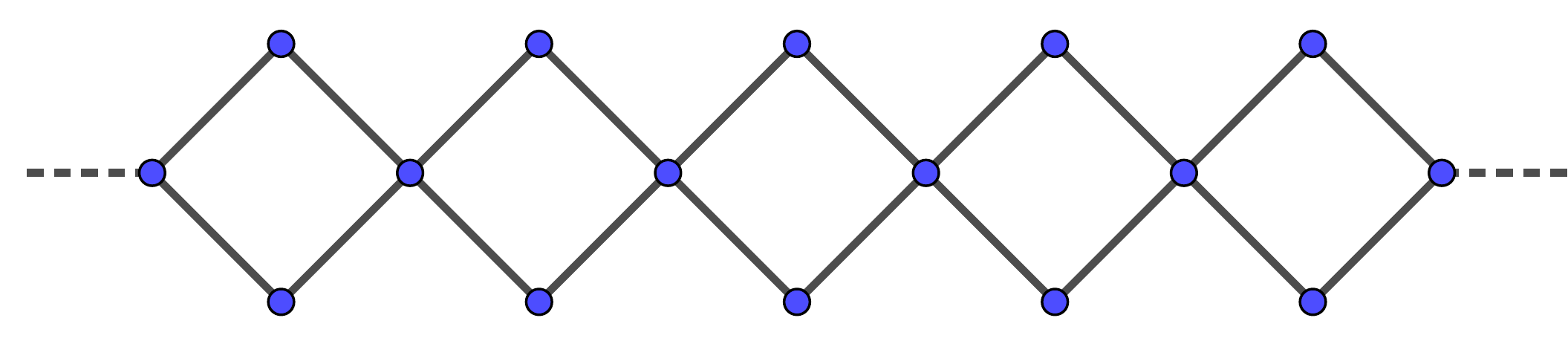}
\caption{A median flat whose isometry group is isomorphic to $\mathbb{Z}_2 \wr D_{\infty}$.}
\label{PL}
\end{center}
\end{figure}

\medskip \noindent
The notion of median flat arises naturally when we study the dynamics of loxodromic isometries of CAT(0) cube complexes, as justified by the next statement:

\begin{prop}\label{prop:AxesFlat}
Let $X$ be a CAT(0) cube complex and $g \in \mathrm{Isom}(X)$ a loxodromic isometry. Fix two points $\zeta, \xi \in \mathfrak{R}X$ such that $g$ admits an axis having $\zeta$ and $\xi$ as points at infinity. The union of all the axes of $g$ having $\zeta$ and $\xi$ as points at infinity is a median flat.
\end{prop}

\noindent
The first step is to show that such a union turns out to be a median subalgebra of $X$, which is essentially a consequence of the following lemma:

\begin{lemma}\label{lem:MedianAxis}
Let $X$ be a CAT(0) cube complex and $g \in \mathrm{Isom}(X)$ a loxodromic isometry. Fix three axes $\alpha, \beta, \gamma$ of $g$ and three vertices $x \in \alpha$, $y \in \beta$, $z \in \gamma$. If $m$ denotes the median point $\mu(x,y,z)$ and if $[m,gm]$ is an arbitrary geodesic between $m$ and $gm$, then $\bigcup\limits_{k \in \mathbb{Z}} g^k [m,gm]$ is an axis of $g$.
\end{lemma}

\begin{proof}
Let $x,y,z \in \mathrm{Med}(g)$ be three vertices, and let $m$ denote the median point of $\{x,y,z\}$. In order to show that $m$ belongs to $\mathrm{Med}(g)$, it is sufficient to prove that $\bigcup\limits_{k \in \mathbb{Z}} g^k [m,gm]$ is a geodesic.

\medskip \noindent
Suppose by contradiction that $\bigcup\limits_{k \in \mathbb{Z}} g^k [m,gm]$ is not a geodesic. So there exist an index $k \in \mathbb{Z} \backslash \{0\}$ and a hyperplane $J$ such that $J$ intersects both $[m,gm]$ and $g^k[m,gm]$. Let $J^-$ and $J^+$ denote the halfspaces delimited by $J$ such that $m \in J^-$ and $gm \in J^+$. By convexity of halfspaces, $J^-$ has to contain at least two vertices among $\{x,y,z\}$ and similarly $J^+$ has to contain at least two vertices among $\{gx,gy,gz\}$. Therefore, we must have $x \in J^-$ and $gx \in J^+$, or $y \in J^-$ and $gy \in J^+$, or $z \in J^-$ and $gz \in J^+$. Without loss of generality, suppose that $x \in J^-$ and $gx \in J^+$. We also know from Lemma \ref{lem:HypSepMedian} that $J$ has to separate $g^kx$ and $g^{k+1}x$, or $g^ky$ and $g^{k+1}y$, or $g^kz$ and $g^{k+1}z$. Notice that $J$ cannot separate $g^kx$ and $g^{k+1}x$ since it already separates $x$ and $gx$ and that $\bigcup\limits_{k \in \mathbb{Z}} g^k [x,gx]$ is a geodesic (for any choice of a geodesic $[x,gx]$ between $x$ and $gx$). Without loss of generality, suppose that $J$ separates $g^ky$ and $g^{k+1}y$. Now, we distinguish two cases.

\medskip \noindent
\textbf{Case 1:} $g^{k+1}y$ belongs to $J^-$ and $g^ky$ to $J^+$. 

\medskip \noindent
So $J$ separates $\{x,g^{k+1}y\}$ and $\{gx,g^ky\}$. Set $N= d( x, g^{k+1}y)+1$. Because, for every $0 \leq j \leq N-1$, the hyperplane $g^jJ$ separates $\{g^jx,g^{k+j+1}y\}$ and $\{g^{j+1}x,g^{k+j}y\}$, it follows that $J, gJ, \ldots, g^{N-1}J$ all separate $g^Nx$ and $g^{k+N+1}y$. Therefore,
$$d\left( x, g^{k+1}y \right) = d \left( g^Nx, g^{k+N+1}y \right) \geq N = d \left( x, g^{k+1}y \right)+1,$$
a contradiction. Let us record what we have just proved for future use:

\begin{fact}\label{fact:TwoDirections}
Let $X$ be a CAT(0) cube complex and $g \in \mathrm{Isom}(X)$ a loxodromic isometry. For every vertices $x,y \in \mathrm{Med}(g)$ and integer $k \in \mathbb{Z}$, no hyperplane can separate  $\{x,g^{k+1}y\}$ and $\{gx,g^ky\}$.
\end{fact}

\noindent
Now let us focus on the second case we have to consider.

\medskip \noindent
\textbf{Case 2:} $g^{k+1}y$ belongs to $J^+$ and $g^ky$ to $J^-$. 

\medskip \noindent
So far, we know that $\{x, y, gy, g^ky \} \subset J^-$ and $\{gx, g^kx,g^{k+1}x, g^{k+1}y \} \subset J^+$. Because $gm \in J^+$, necessarily at least two vertices among $\{gx,gy,gz\}$ have to belong to $J^+$, so that $gz$ has to belong to $J^+$. Since $g^{k+1}x$ and $g^{k+1}y$ belong to $J^+$, necessarily $g^{k+1} m \in J^+$, so that $g^k m \in J^-$. Therefore, at least two vertices among $\{g^kx, g^ky,g^kz\}$ have to belong to $J^-$, hence $g^kz \in J^-$. Thus, we have proved that $J$ separates $gz$ and $g^kz$ with $gz \in J^+$ and $g^kz \in J^-$. Necessarily, $z \in J^+$ and $g^{k+1}z \in J^-$. By applying Fact \ref{fact:TwoDirections} to  $\{x, g^{k+1}z \}$ and $\{gx, gz\}$, we find a contradiction. 

\medskip \noindent
Thus, we have proved that $\gamma:= \bigcup\limits_{k \in \mathbb{Z}} g^k [m,gm]$ is a geodesic, concluding the proof of our lemma.
\end{proof}

\begin{proof}[Proof of Proposition \ref{prop:AxesFlat}.]
Let $F$ denote the union of all the axes of $g$ having $\zeta$ and $\xi$ as points at infinity. We first verify that $F$ is a median subalgebra of $X$. So let $x,y,z \in F$ be three vertices. We know from Lemma \ref{lem:MinMedian} that $\mu(x,y,z)$ belongs to a geodesic $\gamma$ on which $g$ acts by translations. We have
$$\gamma(+ \infty) = \lim\limits_{n \to + \infty} g^n \mu(x,y,z) = \mu \left( \lim\limits_{n\to + \infty} g^nx, \lim\limits_{n \to + \infty} g^ny, \lim\limits_{n \to + \infty} g^nz \right) = \mu( \xi , \xi, \xi)=\xi.$$
Similarly, one shows that $\gamma(- \infty)= \zeta$. Thus, we have proved that $\mu(x,y,z)$ belongs to an axis of $g$ having $\zeta$ and $\xi$ as points at infinity, i.e. $\mu(x,y,z) \in F$.

\medskip \noindent
It remains to show that that $F$ is finite-dimensional. We begin by proving two preliminary claims.

\begin{claim}\label{claim:AxisPassing}
If $x,y \in F$ are adjacent in $X$, then there exists an axis of $g$ passing through both of them.
\end{claim}

\noindent
Fix two axes $\alpha$ and $\beta$ passing through $x$ and $y$ respectively. Let $J$ denote the hyperplane separating $x$ and $y$. Notice that, for every $n \neq 0$, the hyperplane $J$ does not separate $g^nx$ and $g^ny$. Indeed, otherwise $g^n$ would stabilise $J$, as well as the halfspaces it delimits because $g^n$ does not invert any hyperplane, so that $J$ would separate $\{ g^{nk}x \mid k \in \mathbb{Z} \}$ and $\{ g^{nk}y \mid k \in \mathbb{Z} \}$, contradicting the fact that $\alpha$ and $\beta$ have the same points at infinity. On the other hand, once again because $\alpha$ and $\beta$ have the same points at infinity, $J$ necessarily crosses $\alpha$ and $\beta$. The conclusion is that $J$ has to separate $\{ g^{-1}x,g^{-1}y \}$ and $\{gx,gy\}$. Now, two cases may happen: either $J$ separates $x$ from $\{gx,gy\}$ or it separates $y$ from $\{gx,gy\}$. Without loss of generality, we suppose that we are in the former case. 
\begin{figure}
\begin{center}
\includegraphics[trim={0 21cm 33cm 0},clip,scale=0.5]{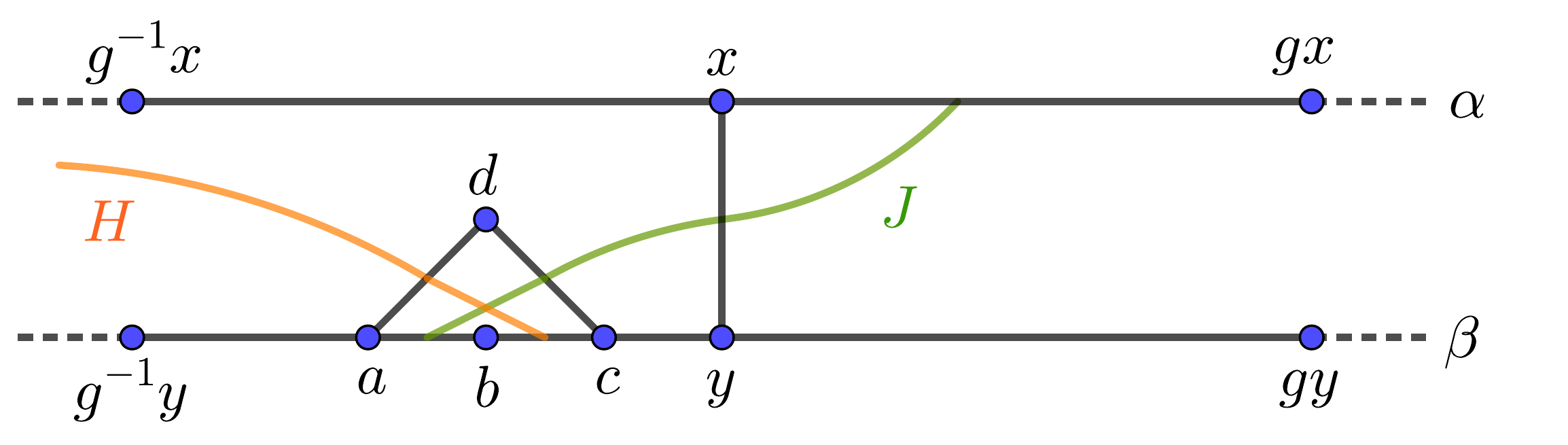}
\caption{Configuration from the proof of Claim \ref{claim:AxisPassing}.}
\label{AxisAdj}
\end{center}
\end{figure}

\medskip \noindent
Let $a,b \in \beta$ denote the endpoints of the edge of $\beta$ crossed by $J$ (so that $b$ is between $a$ and $y$ along $\beta$). Assume that $b \neq y$. Let $c \in \beta$ denote the neighbor of $b$ which is distinct from $a$. See Figure \ref{AxisAdj}. Notice that the hyperplane $H$ separating $b$ and $c$ is transverse to $J$. Indeed, because $H$ does not separate $x$ and $y$ and that $H$ has to cross $\alpha$, it follows that $H$ separates $x$ from some $g^{-k}x$. So $H$ separates $b$ and $c$, and $x$ and $g^{-k}x$, but $J$ separates $\{b,c\}$ and $\{x,g^{-k}x\}$. Therefore, $J$ and $H$ must be transverse. As a consequence, the two edges $[b,a]$ and $[b,c]$ have to span a square. Let $d$ denote the fourth vertex of this square, and let $\delta$ denote the geodesic between $g^{-1}y$ and $y$ obtained from the subsegment $[g^{-1}y,y] \subset \beta$ by replacing $[a,b] \cup [b,c]$ with $[a,d] \cup [d,c]$. The concatenation $\gamma:= \bigcup\limits_{k \in \mathbb{Z}} g^k \delta$ defines an axis of $g$ having the same points at infinity as $\alpha$ and $\beta$. Notice that, by construction, the vertex $y$ still belongs to $\gamma$. Moreover, the edge of $\gamma$ crossed by $J$ is now closer to $y$. 

\medskip \noindent
By iterating this argument to $\beta$, and next to $\alpha$, it follows that there exist two axes $\alpha'$ and $\beta'$ passing through $x$ and $y$ respectively and such that the edges of $\alpha'$ and $\beta'$ crossed by $J$ share an endpoint with $[x,y]$. Because two adjacent edges cannot be crossed by the same hyperplane, it follows that $\alpha'$ and $\beta'$ both pass through $x$ and $y$, concluding the proof of our claim. 

\begin{claim}\label{claim:HypPartition}
Two distinct hyperplanes of $X$ separating at least two vertices of $F$ induce different partitions of $F$.
\end{claim}

\noindent
Let $J$ and $H$ be two hyperplanes of $X$ separating at least two vertices of $F$. Necessarily, $J$ and $H$ separates $\zeta$ and $\xi$. Therefore, if $\alpha$ denotes an axis of $g$ having $\zeta$ and $\xi$ as points at infinity, then $J$ and $H$ have to cross $\alpha$. Let $x \in \alpha$ be a vertex which is between $J$ and $H$ along $\alpha$. Say that $J$ separates $x$ from $\zeta$. Then there exists some sufficiently large $n \geq 1$ such that $H$ separates $x$ from $g^nx$ but $J$ does not separate these two vertices. Consequently, $J$ and $H$ induce distinct partitions of $F$, concluding the proof of our claim.

\medskip \noindent
We are now ready to show that $F$ is finite-dimensional. In fact, we will prove that the cubulation of $F$ is uniformly locally finite. By combining Lemma \ref{lem:TwoCubulations} with Claim~\ref{claim:HypPartition}, it follows that, for every $N \geq 1$, if the cubulation of $F$ contains a vertex with $N$ neighbors, then so does $F$ in $X$. So fix a vertex $x \in F$ and $N$ of its neighbors $x_1, \ldots, x_N$. According to Claim \ref{claim:AxisPassing}, for every $1 \leq i \leq N$, there exists an axis of $g$ passing through $x$; as a consequence, the hyperplane separating $x$ and $x_i$ has to separate $x$ from $gx$. Because two adjacent edges have to be dual to distinct hyperplanes, it follows that $N \leq d(x,gx)= \|g\|$. Thus, we have proved that:

\begin{fact}\label{fact:LocallyFinite}
Every vertex of the cubulation of $F$ admits at most $\|g\|$ neighbors.
\end{fact}

\noindent
A fortiori, the cubulation of $F$ has dimension at most $\|g\|$, concluding the proof of our proposition. 
\end{proof}

\section{Median sets of isometries}\label{section:MedianSets}

\noindent
Our goal in this section is to associate to any isometry a median subalgebra of the cube complex which behaves similarly to minimising sets of isometries in CAT(0) spaces. Our sets are:

\begin{definition}\label{def:MedianSet}
Let $X$ be a CAT(0) cube complex and $g \in \mathrm{Isom}(X)$ an isometry. The \emph{median set} of $g$, denoted by $\mathrm{Med}(g)$, is
\begin{itemize}
	\item the union of all the $d$-dimensional cubes stabilised by $\langle g \rangle$ if $g$ elliptic, where $d$ denotes the minimal dimension of a cube stabilised by $\langle g \rangle$;
	\item the union of all the axes of $g$ if $g$ is loxodromic;
	\item the pre-image under $\pi$ of the union of all the axes of $g$ in $X/ \mathcal{J}$ if $g$ is inverting, where $\mathcal{J}$ is the collection of the hyperplanes inverted by powers of $g$ and where $\pi : X \to X/ \mathcal{J}$ is the canonical map to the cubical quotient $X/ \mathcal{J}$.
\end{itemize}
\end{definition}

\noindent
The fact that the inverting isometry $g$ of $X$ induces a loxodromic isometry of $X/ \mathcal{J}$ is justified by Lemma \ref{lem:MinInX} below. So the median set of an inverting isometry is well-defined. Its definition may seem to be technical, but it turns out to be a very natural set. The reader can keep in mind the example given by Figure \ref{MedianInvEx}. 
\begin{figure}
\begin{center}
\includegraphics[trim={0 21cm 25cm 0},clip,scale=0.45]{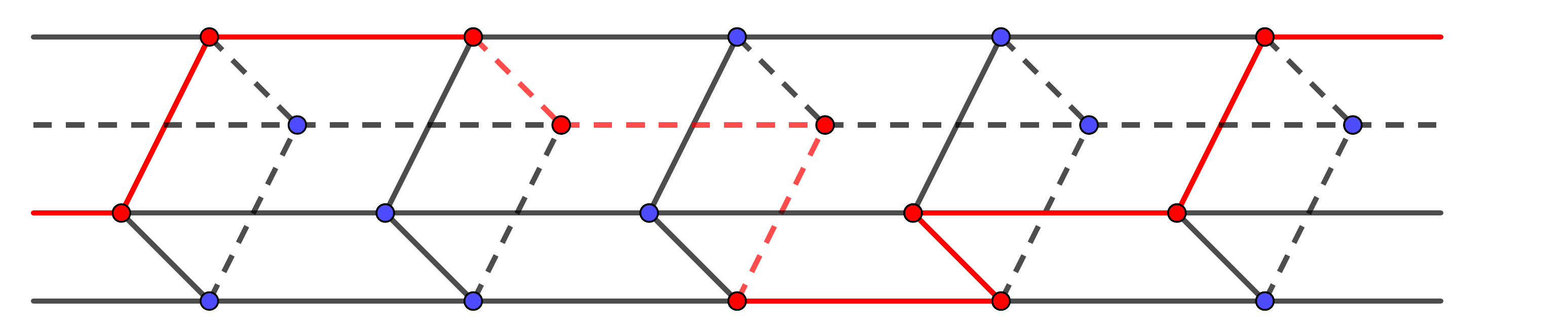}
\caption{The median set $[0,1]^2 \times \mathbb{R}$ of $g= ( \text{rotation}, \ \text{translation} )$.}
\label{MedianInvEx}
\end{center}
\end{figure}

\medskip \noindent
Then, the main result of this section is:

\begin{thm}\label{thm:MedianSet}
Let $G$ be a group acting on a CAT(0) cube complex $X$ and $g \in G$ an isometry such that $\langle g \rangle$ is a normal subgroup of $G$. Then $\mathrm{Med}(g)$ is a median subalgebra of $X$ which is $G$-invariant and which decomposes as a product $T \times F \times Q$ of three median algebras $T,F,Q$ such that:
\begin{itemize}
	\item the action $G \curvearrowright T \times F \times Q$ decomposes as a product of three actions $G \curvearrowright T,F,Q$;
	\item $F$ is a single vertex if $g$ is elliptic, and otherwise it is a median flat on which $g$ acts by translations of length $\lim\limits_{n \to + \infty} \frac{1}{n} d(x,g^nx)>0$;
	\item $Q$ is a finite-dimensional cube, possibly reduced to a single vertex;
	\item $g$ acts trivially on $T$.
\end{itemize}
Moreover, the dimension of $Q$ is zero if $g$ is loxodromic, it coincides with the minimal dimension of a cube of $X$ stabilised by $g$ if $g$ is elliptic, and otherwise it coincides with the number of hyperplanes inverted by powers of $g$.
\end{thm}

\noindent
In order to prove Theorem \ref{thm:MedianSet}, we distinguish three cases, depending on whether the isometry is elliptic, inverting or loxodromic. So Theorem \ref{thm:MedianSet} is the combination of Propositions \ref{prop:MedElliptic}, \ref{prop:MinSet} and \ref{prop:MinInvIsom} below. Before turning to the proofs, we emphasize that, although the median set of a loxodromic isometry $g$ coincides with its minimising set $\mathrm{Med}(g) = \left\{x \in X \mid d(x,gx) = \min \limits_{y \in X} d(y,gy) \right\}$ (see Lemma \ref{lem:MinVsMed}), it turns out that the minimising set of an elliptic or inverting isometry may not be median. See Remarks~\ref{remark:NotMinElliptic} and \ref{remark:NotMinInversing} below. Nevertheless, it can be shown that the median set coincides with the median hull of the minimising set. As this description will not be used in this article, we do not include a proof of this assertion.

\subsection{Elliptic groups of isometries}

\noindent
In this subsection, we prove Theorem \ref{thm:MedianSet} for elliptic isometries. In fact, we prove a more general statement dealing with arbitrary elliptic subgroups:

\begin{prop}\label{prop:MedElliptic}
Let $G$ be a group acting on a CAT(0) cube complex $X$ and $E$ a normal subgroup which has a bounded orbit. If $d$ denotes the minimal dimension of a cube stabilised by $E$, then the union of the $d$-dimensional cubes stabilised by $E$ is a $G$-invariant median subalgebra which decomposes as a product $Q \times Z$ of two median algebras $Q,Z$ such that:
\begin{itemize}
	\item the action $G \curvearrowright Q \times Z$ decomposes as a product of two actions $G \curvearrowright Q, Z$;
	\item $Q$ is a $d$-dimensional cube;
	\item $E$ acts trivially on $Z$.
\end{itemize}
\end{prop}

\begin{proof}
Let $Y$ denote the union of the $d$-dimensional cubes stabilised by $E$. Our first goal is to show that $Y$ is median. We begin by proving an easy but useful observation.

\begin{claim}\label{claim:samehyp}
Two $d$-dimensional cubes stabilised by $E$ are crossed exactly by the same hyperplanes of $X$.
\end{claim}

\noindent
Indeed, if $C_1$ and $C_2$ are two such cubes, then the projection of $C_1$ onto $C_2$ provides a subcube $P \subset C_2$, which has to be $E$-invariant. By minimality of the dimension $d$, it follows that $P=C_2$. It follows from Lemma \ref{lem:HypProjSeparate} that $C_1$ and $C_2$ are crossed by the same hyperplanes of $X$, concluding the proof of our claim.

\medskip \noindent
We are now ready to show that $Y$ is median. Let $x_1,x_2,x_3 \in Y$ be three vertices. For every $i=1,2,3$, there exists a $d$-dimensional cube $C_i$ which is stabilised by $E$ and which contains $x_i$. According to Claim \ref{claim:samehyp}, these cubes are crossed by the same hyperplanes of $X$; let $\mathcal{H}$ denote the collection of these hyperplanes. It follows from Lemma \ref{lem:HypSepMedian} that any hyperplane separating two vertices of $\mu(C_1,C_2,C_3)$ has to cross one the cubes $C_1,C_2,C_3$, i.e. it has to belong to $\mathcal{H}$. As a consequence, because $\mathcal{H}$ is a collection of $d$ pairwise transverse hyperplanes, the convex hull of $\mu(C_1,C_2,C_3)$ has to be a $d$-dimensional cube. As $\mu(C_1,C_2,C_3)$ is clearly $E$-invariant, so is this cube. Thus, we have proved that $\mu(x_1,x_2,x_3)$ belongs to a $d$-dimensional cube stabilised by $E$, i.e. $\mu(x_1,x_2,x_3) \in Y$. 

\medskip \noindent
Now, our goal is to decompose $Y$ as a product. For this purpose, we fix a $d$-dimensional cube $Q$ stabilised by $E$ and a vertex $v \in Q$, and we set 
$$Z = \{ \mathrm{proj}_C(v) \mid \text{$C$ $d$-dimensional cube stabilised by $E$} \}.$$
Notice that any two distinct $d$-dimensional cubes stabilised by $E$ have an empty intersection. Indeed, otherwise their intersection would define a lower dimensional cube stabilised by $E$, contradicting the minimality of $d$. Consequently, for every vertex $x \in Y$, we can denote by $C(x)$ the unique $d$-dimensional cube stabilised by $E$ which contains $x$. We are interested in the map
$$\Phi : \left\{ \begin{array}{ccc} Y & \to & Q \times Z \\ x & \mapsto & \left( \mathrm{proj}_Q(x), \mathrm{proj}_{C(x)}(v) \right) \end{array} \right..$$
First of all, let us verify that $Z$ is a median subalgebra of $X$. Once again, we denote by $\mathcal{H}$ the collection of all the hyperplanes crossing the $d$-dimensional cubes stabilised by $E$. 
\begin{figure}
\begin{center}
\includegraphics[trim={0 19cm 10cm 0},clip,scale=0.45]{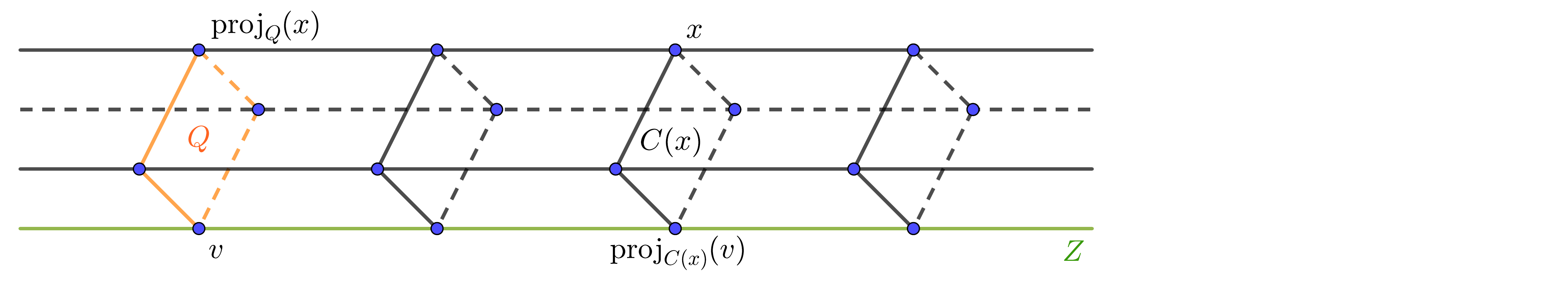}
\caption{The decomposition $Q \times Z$ of $Y$.}
\label{UnionCubeStab}
\end{center}
\end{figure}

\begin{claim}\label{claim:noHinZ}
No hyperplane of $\mathcal{H}$ separates two vertices of $Z$.
\end{claim}

\noindent
Let $a,b \in Z$ be two vertices. By definition of $Z$, one has $a= \mathrm{proj}_{C(a)}(v)$ and $b= \mathrm{proj}_{C(b)}(v)$. As a consequence of Lemma \ref{lem:HypProj}, the hyperplanes separating $v$ from $a$ or $b$ do not cross $C(a)$ or $C(b)$ respectively, i.e. they do not belong to $\mathcal{H}$. Because any hyperplane separating $a$ and $b$ necessarily separates $v$ from either $a$ or $b$, our claim follows.

\medskip \noindent
Now, let $a,b,c \in Z$ be three vertices and let $m$ denote their median point in $Y$. If there exists a hyperplane $J$ separating $m$ from $\mathrm{proj}_{C(m)}(v)$, then it has to belong to $\mathcal{H}$ since the two vertices belong to the cube $C(m)$. But, as a consequence of Claim \ref{claim:noHinZ}, the vertices $a,b,c$ necessarily belong to the halfspace delimited by $J$ which contains $\mathrm{proj}_{C(m)}(v)$, so that $J$ separates $m$ from $a,b,c$, contradicting the convexity of halfspaces. Therefore, no hyperplane separates $m$ and $\mathrm{proj}_{C(m)}(v)$, hence $m = \mathrm{proj}_{C(m)}(v) \in Z$. Thus, we have proved that $Z$ is a median subalgebra of $X$.

\medskip \noindent
The next step is to show that our map $\Phi$ is an isomorphism of median algebras. We begin by showing that $\Phi$ is surjective. So let $q \in Q$ and $z \in Z$ be two vertices. Notice that
$$\Phi \left( \mathrm{proj}_{C(z)}(q) \right) = \left( \mathrm{proj}_Q \left( \mathrm{proj}_{C(z)} (q ) \right), \mathrm{proj}_{C(z)}(v) \right).$$
It follows from Claim \ref{claim:noHinZ} that $\mathrm{proj}_{C(z)}(v)=z$. Also, by applying Lemma \ref{lem:HypProj} twice, we know that no hyperplane separating $q$ and $\mathrm{proj}_{C(z)}(q)$ crosses $C(z)$ and that no hyperplane separating $\mathrm{proj}_{C(z)}(q)$ and $\mathrm{proj}_{Q} (\mathrm{proj}_{C(z)}(q))$ crosses $Q$, so that no hyperplane of $\mathcal{H}$ separates $q$ and $\mathrm{proj}_{Q} (\mathrm{proj}_{C(z)}(q))$. Because these two points belong to $Q$ and that the collection of the hyperplanes crossing $Q$ coincides with $\mathcal{H}$, it follows that $\mathrm{proj}_{Q} (\mathrm{proj}_{C(z)}(q))=q$. Thus, $\Phi( \mathrm{proj}_{C(z)}(q))= (q,z)$, proving the surjectivity of $\Phi$. We record this equality for future use:

\begin{fact}\label{fact:PhiInverse}
The map
$$\Psi : \left\{ \begin{array}{ccc} Q \times Z & \to & Y \\ (q,z) & \mapsto & \mathrm{proj}_{C(z)}(q) \end{array} \right.$$
satisfies $\Phi \circ \Psi = \mathrm{Id}_{Q \times Z}$. 
\end{fact}

\noindent
Now, let $x,y \in Y$ be two vertices. As a consequence of Lemma \ref{lem:HypProjSeparate}, 
$$\mathcal{W}(x,y) \cap \mathcal{H}= \mathcal{W}\left( \mathrm{proj}_Q(x), \mathrm{proj}_Q(y) \right).$$
We also have
$$\mathcal{W}(x,y) \backslash \mathcal{H} = \mathcal{W}(C(x),C(y)) = \mathcal{W} \left(\mathrm{proj}_{C(x)}(v), \mathrm{proj}_{C(y)}(v) \right)$$
where the last equality is justified by Claim \ref{claim:noHinZ}. Therefore,
$$d(x,y) = d \left( \mathrm{proj}_Q(x), \mathrm{proj}_Q(y) \right) + d \left( \mathrm{proj}_{C(x)}(v), \mathrm{proj}_{C(y)}(v) \right).$$
Otherwise saying, if $Y,Q,Z$ are endowed with the metrics induced by $X$, then $\Phi$ is an isometry. It implies that $\Phi$ is injective, but also that $\Phi$ is an isomorphism of median algebras since the median structures of $Y,Q,Z$ are induced by the metric of $X$. 

\medskip \noindent
It remains to study the action of $G$ on $Y$ (which is clearly $G$-invariant because $E$ is a normal subgroup). By using the expression of $\Phi^{-1}$ given by Fact \ref{fact:PhiInverse}, the action of $G$ on $Q \times Z$ is given by
$$g \cdot (q,z) = \left( \mathrm{proj}_Q \left( g \cdot \mathrm{proj}_{C(z)}(q) \right), \mathrm{proj}_{g \cdot C(z)} (v) \right)$$
for every $(q,z) \in Q \times Z$ and $g \in G$. Let us simplify this description.

\begin{claim}\label{claim:ProjProj}
The equality $\mathrm{proj}_Q \left( g \cdot \mathrm{proj}_{C}(q) \right) = \mathrm{proj}_Q(g \cdot q)$ holds for every $d$-dimensional cube $C$ stabilised by $E$, for every $(q,z) \in Q \times Z$, and for every $g \in G$.
\end{claim}

\noindent
It follows from Lemma \ref{lem:HypProj} that no hyperplane separating $q$ from $\mathrm{proj}_{C}(q)$ crosses $C$, so that no hyperplane of $\mathcal{H}$ separates $q$ and $\mathrm{proj}_{C}(q)$. As $\mathcal{H}$ is $G$-invariant, no hyperplane of $\mathcal{H}$ separates $g \cdot q$ and $g \cdot \mathrm{proj}_{C} (q)$ either. We deduce from Lemma \ref{lem:HypProjSeparate} that $g \cdot q$ and $g \cdot \mathrm{proj}_{C} (q)$ have the same projection onto $Q$, proving our claim.

\medskip \noindent
Therefore, the action of $G$ on $Q \times Z$ simplifies as
$$g \cdot (q,z) = \left( \mathrm{proj}_Q \left( g \cdot q \right), \mathrm{proj}_{g \cdot C(z)} (v) \right)$$
for every $(q,z) \in Q \times Z$ and $g \in G$. It is clear that
$$g \cdot z = \mathrm{proj}_{g \cdot C(z)}(v), \ g \in G, z \in Z$$
defines an action $G \curvearrowright Z$, and the fact that
$$g \cdot q =  \mathrm{proj}_Q \left( g \cdot q \right), \ g \in G, q \in Q$$
defines an action $G \curvearrowright Q$ follows from Claim \ref{claim:ProjProj}. Thus, we have proved that the action $G \curvearrowright Q \times Z$ decomposes as a product of actions $G \curvearrowright Q$ and $G \curvearrowright Z$. Notice that, as $E$ stabilises the cube $C(z)$ for every $z \in Z$, necessarily $E$ is included into the kernel of the action $G \curvearrowright Z$. 
\end{proof}

\begin{remark}\label{remark:NotMinElliptic}
It is worth noticing that the median set of an elliptic isometry $g$ cannot be defined as $\mathrm{Min}(g)$. Although it can be proved that $\mathrm{Min}(g)$ is always included into $\mathrm{Med}(g)$, the inclusion can be proper and in fact $\mathrm{Min}(g)$ may not be median. Figure \ref{IsomCube} gives such an example: an isometry of a $3$-cube which acts as a rotation on a $6$-cycle (in blue) and which switches the two remaining vertices (in orange).
\begin{figure}
\begin{center}
\includegraphics[trim={0 18.5cm 48cm 0},clip,scale=0.45]{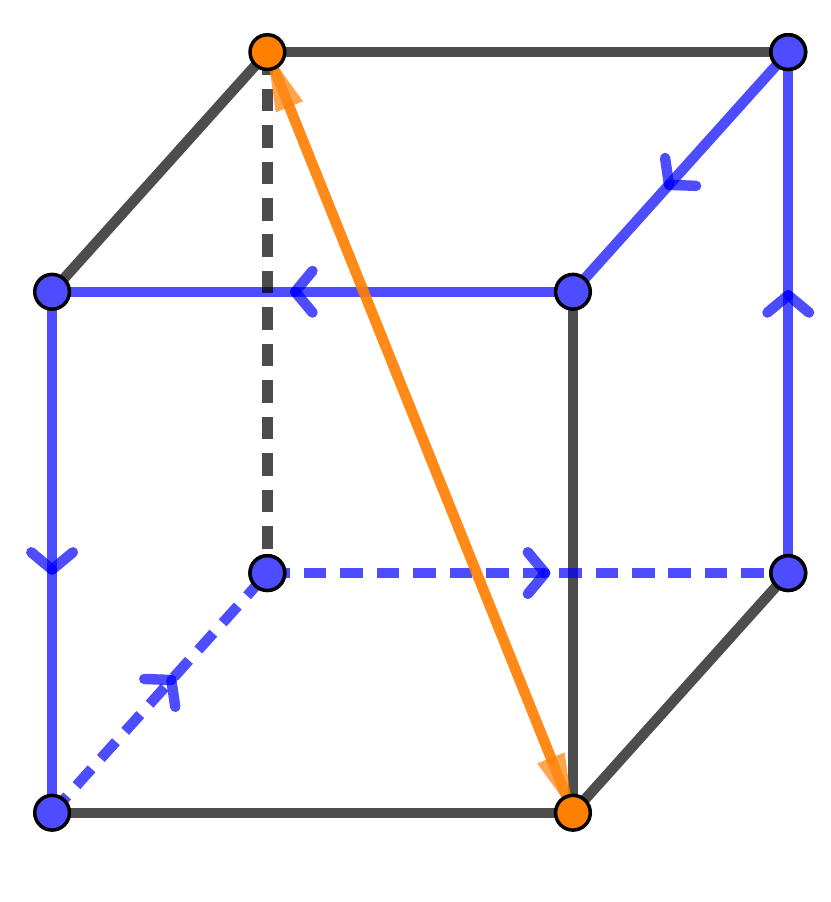}
\caption{An isometry $g$ of a cube such that $\mathrm{Min}(g)$ is not median and $\mathrm{Min}(g) \subsetneq \mathrm{Med}(g)$.}
\label{IsomCube}
\end{center}
\end{figure}
\end{remark}

\subsection{Loxodromic isometries}

\noindent
In this subsection, we prove Theorem \ref{thm:MedianSet} for loxodromic isometries, namely:

\begin{prop}\label{prop:MinSet}
Let $G$ be a group acting on a CAT(0) cube complex $X$ and $g\in G$ a loxodromic isometry such that $\langle g \rangle$ is normal in $G$. Then $\mathrm{Med}(g)$ is a median subalgebra of $X$ which is $G$-invariant and which decomposes as a product $T \times F$ of two median algebras $T,F$ such that:
\begin{itemize}
	\item the action $G \curvearrowright T \times F$ decomposes as a product of two actions $G \curvearrowright T, F$;
	\item $F$ is a median flat on which $g$ acts by translations of length $\|g\|$;
	\item $g$ acts trivially on $T$.
\end{itemize}
\end{prop}

\noindent
In the rest of the section, we will occasionally use the following notation. If $g$ is a loxodromic isometry and $x$ a vertex which belongs to the minimising set of $g$, then the limit $\lim\limits_{k \to \pm \infty} g^k x$ exists in $\overline{X}$ and coincides with the point at infinity $\gamma(\pm \infty)$ if $\gamma$ is an axis of $g$ passing through $x$. For convenience, we may denote this limit by $g^{\pm \infty}x$.

\medskip \noindent
The first step towards the proof of Proposition \ref{prop:MinSet} is to show that the median set of a loxodromic isometry is median, which is a direct consequence of Lemma \ref{lem:MedianAxis}:

\begin{lemma}\label{lem:MinMedian}
Let $X$ be a CAT(0) cube complex and $g \in \mathrm{Isom}(X)$ a loxodromic isometry. Then $\mathrm{Med}(g)$ is median.
\end{lemma}

\noindent
The next step towards the proof of Proposition \ref{prop:MinSet} is to show that the median set of a loxodromic isometry decomposes as a product. Before proving this assertion, we need the following preliminary lemma:

\begin{lemma}\label{lem:HypCrossAxis}
Let $X$ be a CAT(0) cube complex and $g \in \mathrm{Isom}(X)$ a loxodromic isometry. A hyperplane of $X$ crossing an axis of $g$ crosses all the axes of $g$. As a consequence, such a hyperplane separates $\left\{ \lim\limits_{n \to + \infty} g^nx \mid x \in \mathrm{Med}(g) \right\}$ and $\left\{ \lim\limits_{n \to + \infty} g^{-n} x \mid x \in \mathrm{Med}(g) \right\}$. 
\end{lemma}

\begin{proof}
Let $\gamma_1, \gamma_2$ be two axes of $g$ and $J$ a hyperplane intersecting $\gamma_1$. Assume by contradiction that $J$ does not intersect $\gamma_2$. As a consequence, $\gamma_1$ contains a subray $r$ such that $J$ separates $r$ from $\gamma_2$. There exists some $x \in \gamma_1$ such that $r = \{ g^kx \mid k \geq 0\}$ or $r= \{ g^{-k}x \mid k \leq 0\}$. Up to replacing $k$ with $-k$, we may suppose without loss of generality that $r = \{g^kx \mid k \geq 0 \}$. Fix some $y \in \gamma_2$ and set $N=d(x,y)$. Notice that, for every $0 \leq j \leq N$, $g^jJ$ separates $\{ g^kx \mid k \geq j\}$ from $\gamma_2$ so that, in particular, it has to separate $g^Nx$ from $g^Ny$. As $J, gJ, \ldots, g^NJ$ all separate $g^Nx$ and $g^Ny$, it follows that
$$d(x,y)=d \left( g^Nx,g^Ny \right) \geq N+1 = d(x,y)+1,$$
a contradiction. Thus, we have shown that $J$ must intersect $\gamma_2$ as well, proving the first assertion of our lemma. 

\medskip \noindent
In order to prove the second assertion, fix two vertices $x,y \in \mathrm{Med}(g)$ and a hyperplane $J$ intersecting the axes of $g$. The only possibility for $J$ not to separate $g^{-\infty}y$ and $g^\infty x$ is that $J$ separates $\{g^\infty x, g^{- \infty}y \}$ and $\{g^{-\infty} x , g^\infty y \}$. But this is impossible according to Fact \ref{fact:TwoDirections}.
\end{proof}

\begin{cor}\label{cor:DeltaMinus}
Let $X$ be a CAT(0) cube complex and $g \in \mathrm{Isom}(X)$ a loxodromic isometry. If $\gamma_1$ and $\gamma_2$ are two axes of $g$ satisfying $\gamma_1(+ \infty)= \gamma_2(+ \infty)$, then necessarily $\gamma_1(-\infty)= \gamma_2(- \infty)$. 
\end{cor}

\begin{proof}
If there exists a hyperplane $J$ separating $\gamma_1(- \infty)$ and $\gamma_2(\infty)$, then either $J$ separates $\gamma_1(-\infty)$ from the three points $\gamma_1(+ \infty), \gamma_2(-\infty), \gamma_2(+ \infty)$ or it separates $\gamma_2(-\infty)$ from the three points $\gamma_2(+ \infty), \gamma_1(- \infty), \gamma_1(+ \infty)$. In either case, $J$ intersects only one axis among $\gamma_1$ and $\gamma_2$, contradicting Lemma \ref{lem:HypCrossAxis}. Consequently, no hyperplane separates $\gamma_1(- \infty)$ and $\gamma_2(- \infty)$, hence $\gamma_1(- \infty) = \gamma_2(- \infty)$ as desired.
\end{proof}

\noindent
We are now ready to show that the median set of a loxodromic isometry naturally decomposes as a product.
\begin{figure}
\begin{center}
\includegraphics[trim={0 12cm 30cm 0},clip,scale=0.45]{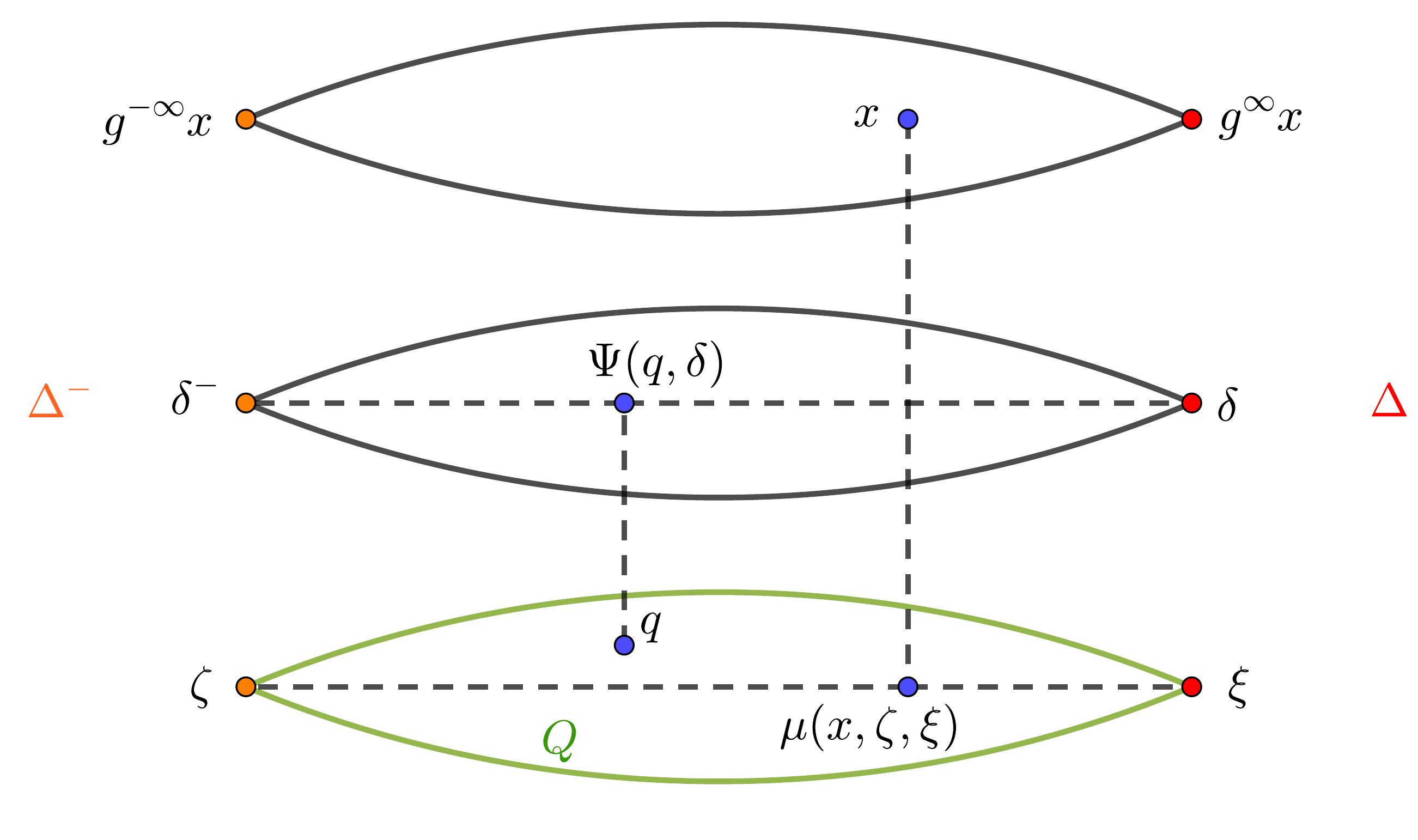}
\caption{The decomposition $T \times F$ of $\mathrm{Med}(g)$.}
\label{MedLoxo}
\end{center}
\end{figure}

\begin{lemma}\label{lem:ProductDeltaF}
Let $X$ be a CAT(0) cube complex and $g \in \mathrm{Isom}(X)$ a loxodromic isometry. Let $T \subset \mathfrak{R}X$ denote $\left\{ \lim\limits_{n \to + \infty}g^nx \mid x \in \mathrm{Med}(g) \right\}$ and let $F$ be the union of all the axes of $g$ having a fixed pair of points at infinity $(\zeta,\xi) \in \mathfrak{R}X \times \mathfrak{R}X$. Then $T$ and $F$ are median subalgebras of $\overline{X}$, and the map
$$\Phi : \left\{ \begin{array}{ccc}  \mathrm{Med}(g) & \to & T \times F \\ x & \mapsto & \left( \lim\limits_{n \to + \infty} g^nx, \mu(x,\zeta,\xi) \right) \end{array} \right.$$
defines an isomorphism of median algebras. 
\end{lemma}

\begin{proof}
First of all, let us notice that $T$ and $F$ are median subalgebras of $\overline{X}$. We already know from Proposition \ref{prop:AxesFlat} that $F$ is a median subalgebra of $X$. So let $A_1,A_2,A_3 \in T$ be three points. So there exist vertices $a_1,a_2,a_3 \in \mathrm{Med}(g)$ such that $\lim\limits_{n \to + \infty} g^na_i = A_i$ for $i=1,2,3$. We have
$$\mu(A_1,A_2,A_3) = \mu \left( \lim\limits_{n \to + \infty} g^n a_1, \lim\limits_{n \to + \infty} g^n a_2, \lim\limits_{n \to + \infty} g^n a_3 \right) = \lim\limits_{n \to + \infty} g^n \mu (a_1,a_2,a_3)$$
where $\mu(a_1,a_2,a_3)$ belongs to $\mathrm{Med}(g)$ according to Lemma \ref{lem:MinMedian}. Therefore, $\mu(A_1,A_2,A_3)$ belongs to $T$.

\medskip \noindent
Now, we want to show our map $\Phi$ is well-defined. More precisely, we need to show that, for every $x \in \mathrm{Med}(g)$, the median point $\mu(x,\zeta,\xi)$ belongs to $F$. The first step is to show that $\mu(x,\zeta,\xi)$ belongs to $\mathrm{Med}(g)$. 

\begin{claim}\label{claim:YmedianInf}
Let $\gamma$ be an axis of $g$. For every $x \in \mathrm{Med}(g)$, the point $\mu(x, \gamma(-\infty), \gamma(+ \infty))$ belongs to $\mathrm{Med}(g)$.
\end{claim}

\noindent
Fix a vertex $y \in \gamma$, and let $N \geq 1$ be sufficiently large so that no hyperplane crossing $\gamma$ separate $x$ from $\{y,g^ny\}$ or $\{y,g^{-n}y\}$ for any $n \geq N$. Notice that, for every $p \geq q \geq N$, the points $\mu(x,g^py,g^{-p}y)$ and $\mu(x,g^qy,g^{-q}y)$ coincides. 

\medskip \noindent
Indeed, if $J$ is a hyperplane separating these two points, it follows from Lemma \ref{lem:HypSepMedian} that $J$ separates $g^py$ and $g^qy$ or $g^{-p}y$ and $g^{-q}y$. Suppose that $J$ separates $g^py$ and $g^qy$, the other case being similar. If $J^+, J^-$ denote the halfspaces delimited by $J$ such that $\mu(x,g^py,g^{-p}y) \in J^+$ and $\mu(x,g^qy,g^{-q}y) \in J^-$, then two cases may happen:
\begin{itemize}
	\item either $g^py \in J^-$ and $g^qy \in J^+$, which is impossible because $g^qy$ and $g^{-q}y$ cannot both belong to $J^+$ as soon as $\mu(x,g^qy,g^{-q}y)$ belongs to $J^-$; 
	\item or $g^py \in J^+$ and $g^q y \in J^-$, so that $y$ and $g^{-p}y$ must belong to $J^-$ (since $\gamma$ is a geodesic) and $x$ has to belong to $J^+$ (since $x$ and $g^{-p}y$ cannot be on the same side of $J$), which implies that $J$ separates $x$ from $\{y,g^q y\}$, a contradiction;
\end{itemize}
Consequently, no hyperplane separates $\mu(x,g^py,g^{-p}y)$ and $\mu(x,g^qy,g^{-q}y)$, whence the equality $\mu(x,g^py,g^{-p}y)=\mu(x,g^qy,g^{-q}y)$. 

\medskip \noindent
We conclude that 
$$\mu(x,\gamma(-\infty),\gamma(+ \infty)) = \lim\limits_{n \to + \infty} \mu \left(x,g^{-n}y,g^ny \right) = \mu \left( x , g^{-N} y , g^Ny \right)$$
belongs to $\mathrm{Med}(g)$ because $\mathrm{Med}(g)$ is median according to Lemma \ref{lem:MinMedian} and because $x, g^{-N}y, g^Ny$ also belong to $\mathrm{Med}(g)$. The proof of Claim \ref{claim:YmedianInf} is complete.

\medskip \noindent
So we know that $\mu(x,\zeta, \xi)$ belongs to $\mathrm{Med}(g)$. As a consequence, there exists a geodesic passing through $\mu(x,\zeta, \xi)$ on which $g$ acts by translations. We have 
$$\gamma(+ \infty)= \lim\limits_{n \to + \infty} g^n \mu(x,\zeta, \xi) = \lim\limits_{n \to + \infty} \mu \left( g^n x , \zeta, \xi \right) = \mu \left( g^{\infty} x ,\zeta, \xi \right).$$
The equality $\xi = \mu( g^{\infty}x,\zeta,\xi)$ follows from the observation that a hyperplane separating $\zeta$ and $\xi$ does not separate $\xi$ and $g^{\infty} x$, as implied by the following immediate consequence of Lemma \ref{lem:HypCrossAxis}:

\begin{fact}\label{fact:HypFandDelta}
Any hyperplane crossing $F$ separates $T$ and $T^- := \left\{ \lim\limits_{n \to + \infty} g^{-n} x \mid x \in \mathrm{Med}(g) \right\}$.
\end{fact}

\noindent
One shows similarly that $\gamma(-\infty)= \lim\limits_{n \to + \infty} g^{-n} \mu(x,\zeta,\xi)=\zeta$. Thus, we have shown that $\mu(x,\zeta,\xi)$ belongs to $F$, proving that our map $\Phi$ is well-defined.

\medskip \noindent
Now, we want to prove that $\Phi$ is surjective. More precisely, we will prove:

\begin{claim}\label{claim:PsiAsPhiInverse}
For every $\delta \in T$, let $\delta^-$ denote the common endpoint in $T^-$ of all the axes $\gamma$ of $g$ satisfying $\gamma(+ \infty)=\delta$. Then the map
$$\Psi : \left\{ \begin{array}{ccc} T \times F & \mapsto & \mathrm{Med}(g) \\ ( \delta, q ) & \mapsto &  \mu \left( q , \delta, \delta^- \right) \end{array} \right..$$
satisfies $\Phi \circ \Psi = \mathrm{Id}_{T \times F}$.  
\end{claim}

\noindent
The uniqueness of $\delta^-$ is justified by Corollary \ref{cor:DeltaMinus}. And the fact that $\Psi$ is well-defined, i.e. that $\mu(q,\delta,\delta^-)$ belongs to $\mathrm{Med}(g)$ for every $(\delta,q) \in T \times F$, follows from Claim \ref{claim:YmedianInf}.

\medskip \noindent
Fix some $(\delta,q) \in T \times F$. We have
$$\Phi \circ \Psi( \delta,q)= \left( g^\infty \mu \left( q, \delta,\delta^- \right), \ \mu \left( \mu \left( q, \delta, \delta^- \right), \zeta, \xi \right) \right).$$
Let $J$ be a hyperplane separating $\zeta$ and $q$. Notice that $J$ does not separate $q$ and $m : = \mu(q, \delta, \delta^-)$, because otherwise it would not separate $\delta$ and $\delta^-$, contradicting Fact~\ref{fact:HypFandDelta}. Therefore, $J$ separates $\zeta$ from $\{ m, \xi\}$. Similarly, one shows that any hyperplane separating $\xi$ and $q$ has to separate $\xi$ from $\{m, \zeta\}$. Finally, let $J$ be a hyperplane separating $m$ and $q$. Then $J$ separates $q$ from $\{\delta, \delta^- \}$, so that $J$ does not separate $\delta$ and $\delta^-$. It follows from Fact \ref{fact:HypFandDelta} that $J$ does not separate $\zeta$ and $\xi$ either. Therefore, $J$ separates $m$ from $\{ \zeta, \xi \}$. Thus, we have proved that $q$ is the median point of $\{m, \zeta, \xi\}$, i.e. 
$$q = \mu \left( \mu \left( q, \delta, \delta^- \right), \zeta ,\xi \right).$$
Next, assume that there exists a hyperplane $J$ separating $\delta$ and $g^\infty m$. As a consequence of Lemma \ref{lem:HypCrossAxis}, $J$ has to separate $\{\delta, \delta^-\}$ and $\{g^\infty m, g^{-\infty} m\}$. But this is impossible since $m$ belongs to geodesics between $g^\infty m$ and $g^{-\infty} m$, and $\delta$ and $\delta^-$. Therefore, no hyperplane separates $\delta$ and $g^\infty m$, proving that $g^\infty m = \delta$.

\medskip \noindent
Thus, we have proved that $\Phi \circ \Psi (\delta,q) = (\delta,q)$ for every $(\delta,q) \in T \times F$, i.e. $\Phi \circ \Psi = \mathrm{Id}_{T \times F}$, concluding the proof of our claim.

\medskip \noindent
It is worth noticing that each of $\mathrm{Med}(g)$, $T$ and $F$ is included into a single cubical component of $\overline{X}$, so that the median structures defined on these subalgebra all come from the metrics defined on the cubical components of $\overline{X}$. As a consequence, in order to show that $\Phi$ defines an isomorphism of median algebras $\mathrm{Med}(g) \to T \times F$, it is sufficient to show that it defines an isometry when $\mathrm{Med}(g)$, $T$ and $F$ are all endowed with the metrics induced by $\overline{X}$.  

\medskip \noindent
So let $x,y \in \mathrm{Med}(g)$ be two vertices. Notice that, as a consequence of Lemma \ref{lem:HypCrossAxis}, we have
$$\mathcal{W}(x,y) \backslash \mathcal{W}(\zeta,\xi) = \mathcal{W}(g^\infty x, g^\infty y).$$
Next, we claim that
$$\mathcal{W}(x,y) \cap \mathcal{W}(\zeta,\xi) = \mathcal{W}( \mu(x, \zeta, \xi), \mu(y, \zeta, \xi)).$$
Indeed, let $J$ be a hyperplane separating $\mu(x,\zeta, \xi)$ and $\mu(y,\zeta,\xi)$. Let $J^-$ denote the halfspace delimited by $J$ which contains the former median point and $J^+$ the halfspace containing the latter point. As $\zeta, \mu(x, \zeta,\xi) \in J^-$ and $\xi \notin J^-$, necessarily $x \in J^-$. Similarly, as $\xi, \mu(y,\zeta,\xi) \in J^+$ and $\zeta \notin J^+$, necessarily $y \in J^+$. Therefore, $J$ has to separate $x$ and $y$ (and it separates $\zeta$ and $\xi$ since $\mu(x,\zeta,\xi)$ and $\mu(y,\zeta,\xi)$ belong to $F$). Conversely, suppose that $J$ is a hyperplane separating both $x$ and $y$ and $\zeta$ and $\xi$. Let $J^-,J^+$ be the halfspaces delimited by $J$ such that $\zeta \in J^-$ and $\xi \in J^+$. We may suppose without loss of generality that $x \in J^-$ and $y \in J^+$. Because $x$ and $\zeta$ both belong to $J^-$, necessarily $\mu(x,\zeta,\xi) \in J^-$; and because $y$ and $\xi$ both belong to $J^+$, necessarily $\mu(y, \zeta, \xi) \in J^+$. Therefore, $J$ separates $\mu(x,\zeta,\xi)$ and $\mu(y, \zeta,\xi)$, concluding the proof of our equality. 

\medskip \noindent
Consequently, 
$$\begin{array}{lcl} d(x,y) & = & \# \mathcal{W}(x,y) = \# \left( \mathcal{W}(x,y) \cap \mathcal{W}(\zeta,\xi) \right) + \# \left( \mathcal{W}(x,y) \backslash \mathcal{W}(\zeta,\xi) \right) \\ \\ & = & \# \mathcal{W}( \mu(x,\zeta,\xi), \mu(y,\zeta,\xi) ) + \# \mathcal{W}( g^\infty x, g^\infty y) \\ \\ & = & d( \mu(x,\zeta,\xi), \mu(y,\zeta,\xi) ) + d( g^\infty x, g^\infty y) \end{array}$$
Thus, we have proved that $\Phi$ is an isometry, concluding the proof of our lemma.
\end{proof}

\begin{proof}[Proof of Proposition \ref{prop:MinSet}.]
We know from Lemma \ref{lem:MinMedian} that $\mathrm{Med}(g)$ is a median subalgebra. Moreover, if we fix some $(\zeta,\xi) \in \mathfrak{R}X \times \mathfrak{R}X$ which are the endpoints of an axis of $g$, if we denote by $F$ the union of all the axes of $g$ with endpoints $(\zeta,\xi)$, and if we set $T = \left\{ \lim\limits_{n \to + \infty} g^nx \mid x \in \mathrm{Med}(g) \right\}$, then we know from Lemma \ref{lem:ProductDeltaF} that $T$ and $F$ are two median subalgebras of $\overline{X}$, that 
$$\Phi : \left\{ \begin{array}{ccc}  \mathrm{Med}(g) & \to & T \times F \\ x & \mapsto & \left( \lim\limits_{n \to + \infty} g^nx, \mu(x,\zeta,\xi) \right) \end{array} \right.$$
defines an isomorphism of median algebras, and we know from Proposition \ref{prop:AxesFlat} that $F$ is a median flat. Notice that $\mathrm{Med}(g)$ is $G$-invariant. Indeed, let $x \in \mathrm{Med}(g)$ be a vertex and $h \in G$ an element. As $\langle g \rangle$ is normal in $G$, there exists some $k \in \mathbb{Z} \backslash \{0\}$ such that $hgh^{-1}=g^k$. But
$$\|g \| = \| hgh^{-1} \| = \| g^k \| = |k| \cdot \|g\|,$$
hence $k=\pm 1$. Therefore,
$$d(g \cdot hx, hx) = d(h^{-1}gh \cdot x, x) =d( g^{\pm 1} \cdot x,x) = d(gx,x) = \min \{ d(y,gy) \mid y \in X \},$$
hence $hx \in \mathrm{Med}(g)$. 

\medskip \noindent
It remains to study the action $G \curvearrowright T \times F$. By using the expression of $\Phi^{-1}$ given by Claim \ref{claim:PsiAsPhiInverse}, we deduce that
$$g \cdot ( \delta, q) = \left( g^\infty h \mu(q,\delta,\delta^-), \ \mu\left( h \mu(q,\delta,\delta^-), \zeta, \xi \right) \right)$$
for every $(\delta,q) \in T \times F$. First, we want to simplify this expression by rewriting the two terms in the right-hand side. For the right term:

\begin{claim}\label{claim:MuMu}
The equality
$$\mu(hq,\zeta,\xi) = \mu( h \mu(q,\delta,\delta^-), \zeta, \xi)$$ 
holds for every $h \in G$ and $(\delta,q) \in T \times F$
\end{claim}

\noindent
Indeed, if there exists a hyperplane $J$ separating these two median points (which both belong to $F$), then $J$ crosses $F$ and it follows from Lemma \ref{lem:HypSepMedian} that $J$ separates $hq$ and $h \mu(q, \delta, \delta^-)$. Since the collection of the hyperplanes crossing $F$ is $G$-invariant, we deduce that there must exist a hyperplane crossing $F$ which separates $q$ and $\mu(q,\delta, \delta^-)$. But such a hyperplane does not separate $\delta$ and $\delta^-$, so that it cannot cross $F$ according to Fact \ref{fact:HypFandDelta}. This complete the proof of our claim. 

\medskip \noindent
For our left term, we need to introduce some notation first. Recall that we set $T^- = \left\{ \lim\limits_{n \to + \infty} g^{-n}x \mid x \in \mathrm{Med}(g) \right\}$. For every $\delta \in T^-$, we define $\delta^+$ as the common endpoint in $T$ of all the axes $\gamma$ of $g$ satisfying $\gamma(- \infty) = \delta$. Such an element is well-defined according to Corollary \ref{cor:DeltaMinus}. For convenience, we set $\delta^+=\delta$ for every $\delta \in T$. Now, we claim that
$$g^\infty h \mu(q,\delta,\delta^+) = (h \delta)^+ \ \text{for every $h \in G$ and $(\delta, q ) \in T \times F$}.$$
Notice that
$$g^\infty h \mu(q,\delta,\delta^-) = \left\{ \begin{array}{cl} hg^\infty \mu(q,\delta,\delta^-) & \text{if $h^{-1}gh=g$} \\ hg^{- \infty} \mu(q,\delta,\delta^-) & \text{if $h^{-1}gh=g^{-1}$} \end{array} \right.$$
and that
$$g^\infty \mu(q, \delta,\delta^-)= \delta \ \text{and} \ g^{- \infty} \mu(q,\delta,\delta^-) = \delta^-.$$
The proof of our second equality is complete. 

\medskip \noindent
Consequently, the action $G \curvearrowright T \times F$ can be now described by
$$h \cdot (\delta, q) = \left( (h\delta)^+, \ \mu(hq,\zeta,\xi) \right) \ \text{for every $h \in G$ and $(\delta,q) \in T \times F$}.$$
It is clear that $h \cdot \delta = (h\delta)^+$, where $h \in G$ and $\delta \in T$, defines an action $G \curvearrowright T$; and the fact that $h \cdot q = \mu(hq, \zeta,\xi)$, where $h \in G$ and $q \in F$, defines an action follows from Claim \ref{claim:MuMu}. Thus, we have proved that the action $G \curvearrowright T \times F$ decomposes as a product of actions $G \curvearrowright T$ and $G \curvearrowright F$. Moreover, it is clear that $g$ acts trivially on $T$ and by translations of length $\|g\|$ on $F$, concluding the proof of our proposition. 
\end{proof}

\subsection{Inversing isometries}

\noindent
In this subsection, we prove Theorem \ref{thm:MedianSet} for inverting isometries, namely:

\begin{prop}\label{prop:MinInvIsom}
Let $G$ be a group acting on a CAT(0) cube complex $X$ and $g \in G$ an inverting isometry such that $\langle g \rangle$ is normal in $G$. Denote by $k$ the number of hyperplanes inverted by powers of $g$. Then $\mathrm{Med}(g)$ is a $G$-invariant median subalgebra which decomposes as a product $T \times F \times Q$ of three median algebras $T,F,Q$ such that:
\begin{itemize}
	\item the action $G \curvearrowright T \times F \times Q$ is a product of three actions $G \curvearrowright T,F,Q$;
	\item $F$ is a median flat on which $g$ acts by translations of length $ \lim\limits_{n \to + \infty} d(x,g^nx)/n$;
	\item $Q$ is a cube of finite dimension $k$;
	\item $g$ acts trivially on $T$.
\end{itemize}
\end{prop}

\noindent
In this decomposition, the cube $Q$ will correspond to the $k$ hyperplanes inverted by the powers of $g$. We begin by showing that there exist only finitely many such hyperplanes and that they are pairwise transverse:

\begin{lemma}\label{lem:InvHypFiniteTransverse}
Let $X$ be a CAT(0) cube complex and $g \in \mathrm{Isom}(X)$ an inverting isometry. There exist only finitely many hyperplanes which are inverted by powers of $g$, and they are pairwise transverse.
\end{lemma}

\begin{proof}
Let $\mathcal{J}$ be the collection of the hyperplanes inverted by powers of $g$. Fix a basepoint $x \in X$ and some hyperplane $J \in \mathcal{J}$. By definition of $\mathcal{J}$, there exists some power $n \in \mathbb{Z}$ such that $g^n$ sends the halfspace delimited by $J$ which contains $x$ to the halfspace delimited by $J$ which does not contain $x$. As a consequence, there exists some $k \in \mathbb{Z}$ such that $g$ sends the halfspace delimited by $g^kJ$ which contains $x$ to the halfspace delimited by $g^{k+1}J$ which does not contain $x$. Necessarily, $g^{k+1}J$ separates $x$ and $gx$. Thus, we have proved that any $\langle g \rangle$-orbit in $\mathcal{J}$ contains a hyperplane separating $x$ and $gx$. Since there exist only finitely many hyperplanes separating $x$ and $gx$, it is sufficient to show that every $\langle g \rangle$-orbit in $\mathcal{J}$ is finite in order to deduce that $\mathcal{J}$ must be finite itself. But, if $J \in \mathcal{J}$ then there exists some power $n \in \mathbb{Z}$ such that $g^n$ inverts $J$; a fortiori, $g^n$ stabilises $J$, so that the $\langle g \rangle$-orbit of $J$ must have cardinality at most $n$, concluding the proof of the first assertion of our lemma.

\medskip \noindent
Now, let $J,H \in \mathcal{H}$ be two hyperplanes. Let $m,n \in \mathbb{Z}$ be two powers such that $g^n$ and $g^m$ invert $J$ and $H$ respectively. If $\ell$ denotes the least common multiple of $m$ and $n$, then $\ell$ cannot be a multiple of both $2m$ and $2n$, since otherwise $\ell/2$ would be a lower common multiple of $m$ and $n$. Say that $\ell$ is not a multiple of $2m$. Because $g^n$ and $g^m$ stabilise $J$ and $H$ respectively, necessarily $g^{\ell}$ stabilises both $J$ and $H$. Also, because $\ell$ is not a multiple of $2m$, then $g^{\ell}$ inverts $H$. It follows that $J$ and $H$ must be transverse, since otherwise $g^{\ell} J$ and $J$ would be contained into two different halfspaces delimited by $H$, contradicting $g^\ell J=J$. Thus, we have proved that any two hyperplanes of $\mathcal{J}$ are transverse, concluding the proof of our lemma. 
\end{proof}

\noindent
Now, we are ready to prove the following statement, which we claimed after Definition~\ref{def:MedianSet} and which is needed to justify that the median set of an inverting isometry is well-defined.

\begin{lemma}\label{lem:MinInX}
Let $X$ be a CAT(0) cube complex and $g \in \mathrm{Isom}(X)$ an inverting isometry. Let $\mathcal{J}$ denote the collection of the hyperplanes of $X$ which are inverted by powers of $g$. Then $g$ defines a loxodromic isometry of the cubical quotient $X/ \mathcal{J}$. 
\end{lemma}

\begin{proof}
First of all, notice that $\mathcal{J}$ is $\langle g \rangle$-invariant, so any element of $\langle g \rangle$ naturally defines an isometry of $X / \mathcal{J}$. In order to show that $g$ defines a loxodromic isometry of $X \backslash \mathcal{J}$, it is sufficient to show that $g$ has an unbounded orbit in $X / \mathcal{J}$ and that its powers do not invert any hyperplane of $X / \mathcal{J}$. The former assertion follows from the facts that $g$ has an unbounded orbit in $X$ and that $\mathcal{J}$ is finite (according to Lemma \ref{lem:InvHypFiniteTransverse}). And the latter assertion follows from the observation that, if a power of $g$ inverts a hyperplane of $X / \mathcal{J}$, then it has to invert the corresponding hyperplane of $X$ as well. 
\end{proof}

\noindent
We need two last preliminary lemmas before turning to the proof of Proposition \ref{prop:MinInvIsom}. The first one is the following:

\begin{lemma}\label{lem:MinInNhyp}
Let $X$ be a CAT(0) cube complex and $g \in \mathrm{Isom}(X)$ an inverting isometry. Denote by $\mathcal{J}=\{J_1, \ldots, J_n \}$ the collection of the hyperplanes which are inverted by powers of $g$. Then $\mathrm{Med}(g) \subset \bigcap\limits_{i=1}^n N(J_i)$. 
\end{lemma}

\begin{proof}
Let $J \in \mathcal{J}$ be a hyperplane and let $x \in X$ be a vertex which does not belong to $N(J)$. As a consequence of Lemma \ref{lem:HypProj}, there exists a hyperplane $H$ separating $x$ from $N(J)$. If $k \in \mathbb{Z}$ is a power such that $g^k$ inverts $J$, then $g^kH$ separates $x$ and $g^kx$. Notice that, as $H$ separates $x$ and $g^kx$, necessarily $g^k H$ also separates $g^kx$ and $g^{2k}x$, so that $H$ crosses twice the path $\gamma := \bigcup\limits_{k \in \mathbb{Z}} g^k [x,gx]$, where $[x,gx]$ is a geodesic between $x$ and $gx$ we fix. 

\medskip \noindent
Notice that, since $\mathcal{J}$ is $\langle g \rangle$-invariant, the map $\pi$ is $\langle g \rangle$-equivariant. Consequently, $\pi([x,gx])$ is a geodesic $[\pi(x), g \pi(x)]$ between $\pi(x)$ and $g \pi(x)$. We have
$$\pi(\gamma)= \pi \left( \bigcup\limits_{k \in \mathbb{Z}} g^k [x,gx] \right) = \bigcup\limits_{k \in \mathbb{Z}} g^k [\pi(x), g \pi(x)].$$
However, since $H$ is not transverse to $J$, we know from Lemma \ref{lem:InvHypFiniteTransverse} that $H$ does not belong to $\mathcal{J}$, so that $H$ defines a hyperplane of $X/ \mathcal{J}$ which crosses $\pi(\gamma)$ twice. As a consequence, $\pi(\gamma)$ is not a geodesic, which implies that $\pi(x)$ does not belong to $\mathrm{Med}_{X/ \mathcal{J}}(g)$, or equivalently, that $x$ does not belong to $\mathrm{Med}(g)$. 

\medskip \noindent
Thus, we have proved that, if $x$ belongs to $\mathrm{Med}(g)$, then it has to belong to $N(J)$ for every $J \in \mathcal{J}$.
\end{proof}

\noindent
Finally, our last preliminary lemma is:

\begin{lemma}\label{lem:SectionPi}
Let $X$ be a CAT(0) cube complex and $g \in \mathrm{Isom}(X)$ an inverting isometry. Let $\mathcal{J} = \{J_1, \ldots, J_n\}$ denote the collection of the hyperplanes of $X$ which are inverted by powers of $g$, and $\pi : X \to X/ \mathcal{J}$ the canonical map from $X$ to the cubical quotient $X/ \mathcal{J}$. If we fix an arbitrary halfspace $D_i$ delimited by $J_i$ for every $1 \leq i \leq n$, then $\pi$ induces an isometry $\mathrm{Med}_X(g) \cap \bigcap\limits_{i=1}^n D_i \to \mathrm{Med}_{X / \mathcal{J}} (g)$ when $\mathrm{Med}_X(g)$ and $\mathrm{Med}_{X / \mathcal{J}}(g)$ are endowed with the metrics induced by $X$ and $X / \mathcal{J}$ respectively.
\end{lemma}

\begin{proof}
Because no two vertices of $\bigcap\limits_{i=1}^n D_i$ are separated by a hyperplane of $\mathcal{J}$, it follows from Lemma \ref{lem:DistQuotient} that $\pi$ induces an isometric embedding $\bigcap\limits_{i=1}^n D_i \to X/ \mathcal{J}$. It remains to show that, for every $z \in \mathrm{Med}_{X/\mathcal{J}}(g)$, there exists some $x \in \mathrm{Med}_X(g) \cap \bigcap\limits_{i=1}^n D_i$ such that $\pi(x)=z$. 

\medskip \noindent
By definition of $\mathrm{Med}_X(g)$, there exists some $a \in \mathrm{Med}_X(g)$ such that $\pi(a)=z$. Let $x$ denote the projection of $a$ onto $\bigcap\limits_{i=1}^n D_i$. According to Lemma \ref{lem:HypProj}, the hyperplanes separating $x$ and $a$ have to separate $a$ from $\bigcap\limits_{i=1}^n D_i$. Since $a$ belongs to $\bigcap\limits_{i=1}^n N(J_i)$ according to Lemma \ref{lem:MinInNhyp}, it follows that $x$ and $a$ are only separated by hyperplanes of $\mathcal{J}$. Therefore, $\pi(x)=\pi(a)=z$. We record what we have just proved for future use:

\begin{fact}\label{fact:InversePhi}
Let $z \in \mathrm{Med}_{X/ \mathcal{J}}(g)$ be a vertex. For every $a \in \pi^{-1}(z)$, we have
$$\pi \left( \mathrm{proj}_D(a) \right) = z \ \text{and} \ \mathrm{proj}_D(a) \in \mathrm{Med}_X(g) \cap D$$
where $D:= \bigcap\limits_{i=1}^n D_i$. 
\end{fact}

\noindent
The proof of our lemma is complete. 
\end{proof}

\begin{proof}[Proof of Proposition \ref{prop:MinInvIsom}.]
Let $\mathcal{J}= \{J_1, \ldots, J_n\}$ be the collection of the hyperplanes of $X$ which are inverted by powers of $g$. Notice that $\mathrm{Med}_X(g)$ is $G$-invariant. Indeed, the fact that $\langle g \rangle$ is a normal subgroup of $G$ implies that $\mathcal{J}$ is $G$-invariant, so that $G$ naturally acts on $X/ \mathcal{J}$ and $\pi : X \to X/ \mathcal{J}$ is $G$-equivariant. Moreover, we know from Proposition \ref{prop:MinSet} that $\mathrm{Med}_{X/ \mathcal{J}}(g)$ must be $G$-invariant. Next, notice that $\mathrm{Med}_X(g)$ is a median subalgebra of $X$. Indeed, let $x,y,z \in \mathrm{Med}_X(g)$ be three vertices and let $m:=\mu(x,y,z)$ denote their median point. Then $\pi(m)$ is the median point of $\pi(x)$, $\pi(y)$ and $\pi(z)$. It follows from Lemmas \ref{lem:MinMedian} that $m \in \pi^{-1}(\pi(m)) \subset \pi^{-1} \left( \mathrm{Med}_{X/ \mathcal{J}}(g) \right) = \mathrm{Med}_X(g)$, concluding the proof of our claim. 
\begin{figure}
\begin{center}
\includegraphics[trim={0 15.5cm 28cm 0},clip,scale=0.45]{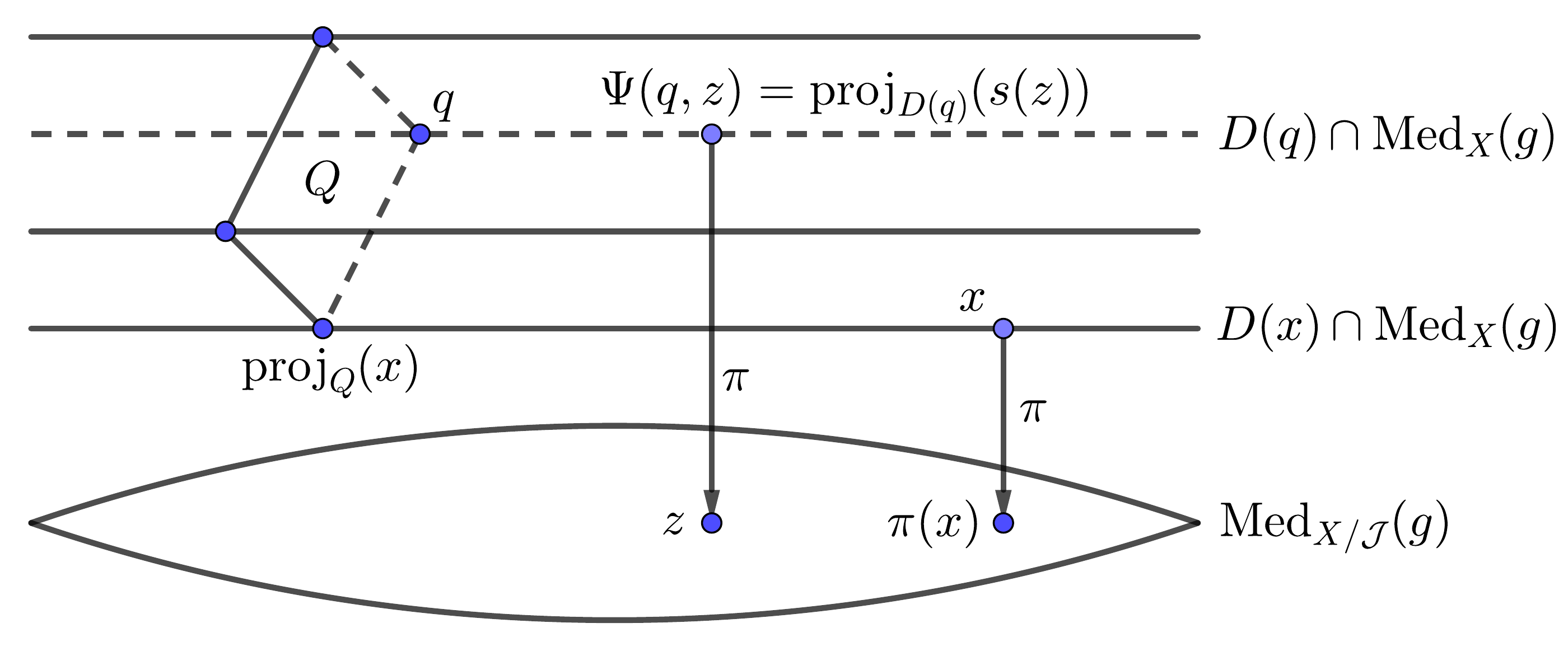}
\caption{The decomposition $Q \times \mathrm{Med}_{X/ \mathcal{J}}(g)$ of $\mathrm{Med}_X(g)$.}
\label{MedInv}
\end{center}
\end{figure}

\medskip \noindent
According to Lemma \ref{lem:InvHypFiniteTransverse}, the hyperplanes of $\mathcal{J}$ are pairwise transverse, so there exists a cube $Q \subset X$ such that $\mathcal{J}$ coincides with the collection of the hyperplanes crossing $Q$. Our goal now is to show that the map
$$\Phi : \left\{ \begin{array}{ccc} \mathrm{Med}_X(g) & \to & Q \times \mathrm{Med}_{X/ \mathcal{J}}(g) \\ x & \mapsto & \left( \mathrm{proj}_Q(x), \pi(x) \right) \end{array} \right..$$
is an isomorphism of median algebras. We begin by showing that $\Phi$ is surjective. For every $q \in Q$, let $D(q)$ denote the intersection of all the halfspaces delimited by the hyperplanes of $\mathcal{J}$ which contain $q$. Also, fix a section $s : X/ \mathcal{J} \to X$ of $\pi$. Then we claim that
$$\Psi: \left\{ \begin{array}{ccc} Q \times \mathrm{Med}_{X/ \mathcal{J}}(g) & \to & \mathrm{Med}_X(g) \\ (q,z) & \mapsto & \mathrm{proj}_{D(q)} (s(z)) \end{array} \right.$$
satisfies $\Phi \circ \Psi = \mathrm{Id}_{Q \times \mathrm{Med}_{X/ \mathcal{J}}(g)}$. Indeed, for every $(q,z) \in Q \times \mathrm{Med}_{X/ \mathcal{J}}(g)$, we have $\pi( \mathrm{proj}_{D(q)}(s(z)))=z$ according to Fact \ref{fact:InversePhi}; and we also have $\mathrm{proj}_Q( \mathrm{proj}_{D(q)}(s(z))) = q$ because by construction $q$ and $\mathrm{proj}_{D(q)}(s(z))$ are not separated by any hyperplane of $\mathcal{J}$. Thus, the surjectivity of $\Phi$ is proved. 

\medskip \noindent
Notice that, because the median structures of $Q$ and $\mathrm{Med}_{X/ \mathcal{J}}(g)$ come from the metrics of $X$ and $X/ \mathcal{J}$ respectively, it is sufficient to show that $\Phi$ is an isometry, when $\mathrm{Med}_X(g)$, $Q$ and $\mathrm{Med}_{X/ \mathcal{J}}(g)$ are endowed with the metrics induced by $X$ and $X/ \mathcal{J}$, in order to deduce that $\Phi$ defines an isomorphism of median algebras. 

\medskip \noindent
So let $x,y \in \mathrm{Med}_X(g)$ be two vertices. For every $1 \leq i \leq n$, let $D_i$ denote the halfspace delimited by $J_i$ which contains $x$. Also, let $y'$ denote the projection of $y$ onto $D:= \bigcap\limits_{i=1}^n D_i$. Because no hyperplane of $\mathcal{J}$ separates $x$ and $y'$, it follows that $x$ and $y'$ have the same projection onto $Q$. Moreover, it follows from Lemmas \ref{lem:HypProj} and \ref{lem:MinInNhyp} that the hyperplanes separating $y$ and $y'$ all belong to $\mathcal{J}$. Consequently,
$$d(y,y') = d \left( \mathrm{proj}_Q(y), \mathrm{proj}_Q(y') \right)=d \left( \mathrm{proj}_Q(y), \mathrm{proj}_Q(x) \right).$$
Next, once again because the hyperplanes separating $y$ and $y'$ all belong to $\mathcal{J}$, we also know that $\pi(y)=\pi(y')$. Therefore, since $x$ and $y'$ both belong to $D$, it follows from Lemma \ref{lem:SectionPi} that 
$$d(x,y')=d(\pi(x),\pi(y'))= d( \pi(x), \pi(y)).$$
We conclude that 
$$d(x,y)=d(x,y')+d(y',y) = d \left( \mathrm{proj}_Q(y), \mathrm{proj}_Q(x) \right) + d( \pi(x), \pi(y)),$$
concluding the proof of our claim. 

\medskip \noindent
Now, let us focus on the action $G \curvearrowright Q \times \mathrm{Med}_{X/ \mathcal{J}}(g)$. By using the expression of $\Psi = \Phi^{-1}$ given above, this action can be described as
$$h \cdot (q,z) = \left( \mathrm{proj}_Q \left( h \cdot \mathrm{proj}_{D(q)}(s(z)) \right), \ \pi \left( h \cdot \mathrm{proj}_{D(q)} (s(z)) \right) \right)$$
for every $(q,z) \in Q \times \mathrm{Med}_{X/ \mathcal{J}}(g)$. We would like to simplify this expression. 

\medskip \noindent
First, notice that, as a consequence of Lemmas \ref{lem:HypProj} and \ref{lem:MinInNhyp}, the vertices $s(z)$ and $\mathrm{proj}_{D(q)}(s(z))$ are only separated by hyperplanes of $\mathcal{J}$, so they must have the same image under $\pi$. Consequently,
$$ \pi \left( h \cdot \mathrm{proj}_{D(q)} (s(z)) \right)=  h \cdot \pi \left( \mathrm{proj}_{D(q)} (s(z)) \right)= h \cdot \pi (s(z)) = h \cdot z.$$
Next, notice that no hyperplane of $\mathcal{J}$ separates $q$ and $\mathrm{proj}_{D(q)}(s(z))$. Because $\mathcal{J}$ is $G$-invariant (since $\langle g \rangle$ is a normal subgroup of $G$), it follows that no hyperplane of $\mathcal{J}$ separates $h \cdot q$ and $h \cdot \mathrm{proj}_{D(q)}(s(z))$ either. Hence
$$\mathrm{proj}_Q \left( h \cdot \mathrm{proj}_{D(q)}(s(z)) \right) = \mathrm{proj}_Q(hq).$$
Thus, we have proved that the action $G \curvearrowright Q \times \mathrm{Med}_{X/ \mathcal{J}}(g)$ can be described as
$$ h \cdot (q,z) = \left( \mathrm{proj}_Q(h \cdot q), \ h \cdot z \right) \ \text{for every $(q,z) \in Q \times \mathrm{Med}_{X/ \mathcal{J}}(g)$}.$$
Notice that
$$\mathrm{proj}_Q \left( h_1 \cdot \mathrm{proj}_Q \left( h_2 \cdot q \right) \right) = \mathrm{proj}_Q \left( h_1h_2 \cdot q \right) \ \text{for every $h_1,h_2 \in G$ and $q \in Q$}.$$
Indeed, $h_2q$ and $\mathrm{proj}_Q(h_2q)$ are not separated by any hyperplane of $\mathcal{J}$. As $\mathcal{J}$ is $G$-invariant, it follows that $h_1h_2q$ and $h_1 \mathrm{proj}_Q(h_2q)$ are not separated by any hyperplane of $\mathcal{J}$ as well, showing that these two vertices have the same projection onto $Q$. 

\medskip \noindent
Thus, we have proved that the action $G \curvearrowright Q \times \mathrm{Med}_{X/ \mathcal{J}}(g)$ decomposes as a product of two actions $G \curvearrowright Q, \mathrm{Med}_{X/ \mathcal{J}}(g)$. It is worth noticing that our action $G \curvearrowright \mathrm{Med}_{X/ \mathcal{J}}(g)$ coincides with the action on $\mathrm{Med}_{X/ \mathcal{J}}(g)$ induced by $G \curvearrowright X/ \mathcal{J}$. 

\medskip \noindent
By applying Proposition \ref{prop:MinSet} to the action of $G$ on $X/ \mathcal{J}$ and to the element $g$, we find that $\mathrm{Med}_{X/ \mathcal{J}}(g)$ decomposes as a product of median algebras $T \times F$ such that:
\begin{itemize}
	\item the action $G \curvearrowright T \times F$ decomposes as a product of two actions $G \curvearrowright T,F$;
	\item $L$ is a median flat on which $g$ acts by translations of length $\|g\|_{X/ \mathcal{J}}$;
	\item $g$ acts trivially on $T$.
\end{itemize}
Notice that, as a consequence of Lemma \ref{lem:DistQuotient}, we have
$$d_{X/ \mathcal{J}}(\pi(x), g \pi(x)) \leq d_X(x,gx) \leq d_{X/ \mathcal{J}}( \pi(x),g \pi(x)) + \# \mathcal{J}$$
for every vertices $x,y \in X$. Consequently, 
$$\|g \|_{X/ \mathcal{J}} = \lim\limits_{n \to + \infty} \frac{1}{n} d_{X/ \mathcal{J}}(x,g^nx) = \lim\limits_{n \to + \infty} \frac{1}{n} d_{X}(x,g^nx).$$
Therefore, the decomposition $Q \times T \times F$ of the median set $\mathrm{Med}_X(g)$ is the decomposition we were looking for.
\end{proof}

\begin{remark}\label{remark:NotMinInversing}
It is worth noticing that the median set of an inverting isometry $g$ cannot be defined as $\mathrm{Min}(g)$, since the minimising set may not be median. Indeed, let $g=(g_1,g_2)$ be the isometry of $[0,1]^2 \times \mathbb{R}$ such that $g_1$ is the isometry of $[0,1]^2$ given by Figure \ref{IsomCube} and $g_2$ a translation of length one. Then $\mathrm{Min}(g)= \mathrm{Min}(g_1) \times \mathbb{R}$ is not median because $\mathrm{Min}(g_1)$ is not a median subset of the square.
\end{remark}

\section{Applications}\label{section:applications}

\subsection{Cubulating centralisers}

\noindent
This section is dedicated to the proof of the following statement, which improves the cubulation of centralisers proved in \cite{HaettelArtin}:

\begin{thm}\label{thm:centralisers}
Let $G$ be a group acting geometrically on a CAT(0) cube complex $X$. For every infinite-order element $g \in G$,  the centraliser $C_G(g)$ also acts geometrically on a CAT(0) cube complex, namely the cubulation of the median subalgebra $\mathrm{Med}_{X'}(g)$ of the barycentric subdivision $X'$ of $X$.
\end{thm}

\noindent
Our theorem will be essentially a consequence of the following two preliminary lemmas:

\begin{lemma}\label{lem:MinVsMed}
Let $X$ be a CAT(0) cube complex and $g \in \mathrm{Isom}(X)$ a loxodromic isometry. Then $\mathrm{Min}(g)= \mathrm{Med}(g)$. 
\end{lemma}

\noindent
This lemma is a direct consequence of \cite[Corollary 6.2]{HaglundAxis}.

\begin{lemma}\label{lem:GeomActionMin}
Let $G$ be a group acting geometrically on a CAT(0) cube complex $X$ and $g \in G$ a loxodromic isometry. Then $C_G(g)$ acts on $\mathrm{Min}(g)$ with finite stabilisers and with finitely many orbits.
\end{lemma}

\begin{proof}
Because $G$ acts properly on $X$, it is clear that $C_G(g)$ acts on $\mathrm{Min}(g)$ with finite stabilisers. Now, suppose by contradiction that $C_G(g)$ acts on $\mathrm{Min}(g)$ with infinitely many orbits. Fix a collection of vertices $x_0, x_1, \ldots \in \mathrm{Min}(g)$ which belong to pairwise distinct $C_G(g)$-orbits. Because $G$ acts cocompactly on $X$, there exists some constant $C \geq 0$ such that, for every $i \geq 1$, there exists some $g_i \in G$ satisfying $d(x_0,g_i x_i) \leq C$. Because $X$ is locally finite, up to taking a subsequence, we may suppose without loss of generality that $g_ix_i = g_jx_j$ for every $i,j \geq 1$. Next, notice that
$$d(x_0,g_igg_i^{-1} x_0) = d( g_i^{-1}x_0, gg_i^{-1}x_0) \leq d(x_i,gx_i) +2C = \| g\| +2C.$$
Because $X$ is locally finite and since its vertex-stabilisers are finite, up to taking a subsequence, we may suppose without loss of generality that $g_ig g_i^{-1}= g_jgg_j^{-1}$ for every $i,j \geq 1$. Fixing two distinct indices $i,j \geq 1$, we have
$$C_G(g) \cdot x_j = C_G(g) \cdot g_i^{-1}g_j x_j = C_G(g) \cdot g_i^{-1} g_ix_i = C_G(g) \cdot x_i$$
because $g_i^{-1}g_j$ belongs to $C_G(g)$, contradicting the fact that $x_i$ and $x_j$ have distinct $C_C(g)$-orbits. 
\end{proof}

\begin{proof}[Proof of Theorem \ref{thm:centralisers}.]
Up to taking the barycentric subdivision of $X$, we may suppose without loss of generality that $g$ is a loxodromic isometry of $X$. We know from Proposition \ref{prop:MinSet} that $\mathrm{Med}(g)$ is a median subalgebra of $X$ which is $C_G(g)$-invariant, so $C_G(g)$ naturally acts on the cubulation $C(\mathrm{Med}(g))$. By combining Lemmas \ref{lem:MinVsMed} and \ref{lem:GeomActionMin}, it follows that $C_g(G)$ acts on $C(\mathrm{Med}(g))$ with finite vertex-stabilisers and with finitely many orbits of vertices. Because $C(\mathrm{Med}(g))$ is locally finite according to Fact \ref{fact:LocallyFinite}, we conclude that $C_G(g)$ acts geometrically on the CAT(0) cube complex $C(\mathrm{Med}(g)$.
\end{proof}

\subsection{A splitting theorem}\label{section:splitting}

\noindent
In this section, we show that normal abelian subgroups in groups acting properly on CAT(0) cube complexes are direct factors up to finite index. Compare with \cite[Theorem II.6.12]{MR1744486} in the CAT(0) setting. 

\begin{thm}\label{thm:splitting}
Let $G$ be a group acting on a CAT(0) cube complex $X$ and $A \lhd G$ a normal finitely generated subgroup. Assume that a non-trivial element of $A$ is never elliptic. Then $A$ is a direct factor of some finite-index subgroup of $G$.
\end{thm}

\noindent
We begin by observing that the normal subgroups under consideration in Theorem \ref{thm:splitting} are central in a finite-index subgroup.

\begin{prop}\label{prop:Centralising}
Let $G$ be a group acting on a CAT(0) cube complex $X$ and $A \leq G$ a normal finitely generated abelian subgroup. Assume that the only element in $A$ with bounded orbits is the trivial element. Then there exists a finite-index subgroup in $G$ which contains $A$ and which centralises it.
\end{prop}

\noindent
The proof of the proposition lies on the following lemma, which will be also used in Section \ref{section:FTT}.

\begin{lemma}\label{lem:AbelianFlat}
Let $A$ be a finitely generated abelian group acting on a CAT(0) cube complex $X$. Then $X$ contains a median subalgebra which is $A$-invariant and which decomposes as a product of a finite-dimensional cube which is a median flat or a single point.
\end{lemma}

\begin{proof}
We argue by induction over the rank of $A$, i.e. the minimal cardinality of a generating set. If $A$ has rank zero, then it suffices to take an arbitrary point in $X$. From now on, assume that $A$ has rank at least one. Write $A$ as $A' \oplus \langle a \rangle$ such that the rank of $A'$ is smaller than the rank of $A$. According to Theorem \ref{thm:MedianSet}, there exists a median subalgebra $Y'$ (namely, the median set of $a$) which is $A$-invariant and which decomposes as a product $T' \times Q' \times F'$ of median algebras $T',Q',F'$ such that:
\begin{itemize}
	\item $A \curvearrowright Y'$ decomposes as a product of actions $A \curvearrowright T',Q',F'$;
	\item $Q'$ is a finite-dimensional cube and $F'$ is a median flat or a single point;
	\item $a$ acts trivially on $T'$.
\end{itemize}
The latter point implies that the action $A \curvearrowright T'$ factorises through $A'$. We know by induction that the action $A' \curvearrowright T'$ preserves a submedian algebra $Y'' \subset T'$ which decomposes as a product $Q'' \times F''$ where $Q''$ is a finite-dimensional cube and where $F''$ is a median flat or a single point. We conclude that $A$ preserves the median subalgebra $Y'' \times Q' \times F' \subset Y'$, which decomposes as the product $(Q' \times Q'') \times (F' \times F'')$ of a finite-dimensional cube $Q' \times Q''$ with $F' \times F''$, which is a median flat or a single point according to Lemma \ref{lem:ProductFlats}. 
\end{proof}

\begin{proof}[Proof of Proposition \ref{prop:Centralising}.]
There is nothing to prove if $A$ is trivial, so from now on we assume that $A$ is non-trivial. As a consequence of Lemma \ref{lem:AbelianFlat}, $A$ preserves a median subalgebra $Y$ which decomposes as a product $Q \times F$ where $Q$ a finite-dimensional cube and $F$ a median flat or a single point. Notice that, if $a \in A$ is is elliptic in $Y$, then it has a finite orbit in $Y$ and so in $X$, hence $a=1$. In other words, $A$ also acts on $Y$ in such a way that the only element with bounded orbits is the trivial element.

\medskip \noindent
Fix an $R \geq 0$ and set $S:= \{ a \in A \mid |a|_X \leq R\}$ where $|\cdot|_X : a \mapsto \lim\limits_{n \to + \infty} \frac{1}{n} d_X(x,a^nx)$ for some fixed $x \in X$. Here we denote by $d_X$ the metric of $X$ in order to avoid confusion with the metric $d_Y$ defined on $Y$ when thought of as a CAT(0) cube complex on its own. Because $d_Y \leq d_X$ as a consequence of Lemma \ref{lem:TwoCubulations}, and because $d_Y^2 \leq d_Y$ where $d_Y^2$ denotes the CAT(0) metric defined on (the cubulation of) $Y$, we have the inclusion $S \subset S' := \{a \in A \mid |a|_Y^2 \leq R\}$ where $|\cdot |_Y^2 : a \mapsto \lim\limits_{n \to + \infty} \frac{1}{n} d_Y^2(x,a^nx)$. It follows from the classical flat torus theorem \cite[Theorem II.7.20]{MR1744486} that $S'$ (and a fortiori $S$) must be finite.

\medskip \noindent
Thus, we have proved that $|\cdot |_X$ defines a proper function on $A$. Moreover, it is also invariant under the action of $G$ on $A$ by conjugations. Therefore, if we fix a basis $a_1, \ldots, a_n \in A$ and if we set $M:= \max \{ |a_i |_X \mid 1 \leq i \leq n\}$, then $G$ acts on the finite set $\{a \in A \mid |a|_X \leq M\}$ and the kernel of the action provides a finite-index subgroup of $G$ which contains $A$ and which centralises it. 
\end{proof}

\begin{proof}[Proof of Theorem \ref{thm:splitting}.]
Fix a non-trivial element $g \in A$. Up to subdividing $X$, we may suppose without loss of generality that $g$ is a loxodromic element. Because $g$ is central in $G$, it follows from Proposition \ref{prop:MinSet} that $G$ acts on a median subalgebra $T \times F \subset X$ such that: 
\begin{itemize}
	\item the action $G \curvearrowright T \times F$ decomposes as a product of two actions $G \curvearrowright T,F$;
	\item $F$ is a median flat on which $g$ acts by translations of positive length. 
\end{itemize}
As a consequence of Proposition \ref{prop:Busemann} there exist a finite-index subgroup $G_0 \leq G$ and a morphism $\varphi : G_0 \to \mathbb{Z}$ such that $\varphi(g) \neq 0$. Fix an element $a \in A$ such that $\varphi(a)$ generates the image $\varphi(A)$, and set $H:= \varphi^{-1}(\varphi(A))$. Notice that $H$ is a finite-index subgroup of $G$. Because $a$ is central in $G$, it follows that $H= H_0 \oplus \langle a \rangle$ and $A= A_0 \oplus \langle a \rangle$ where $H_0= \mathrm{ker}(\varphi) \cap H$ and $A_0 = \mathrm{ker}(\varphi) \cap A$. Notice that the rank of $A_0$ is smaller than the rank of $A$, so by arguing by induction we may suppose without loss of generality that $H_0$ contains a finite-index subgroup which decomposes as a direct sum $A_0 \oplus K$. Then $A_0 \oplus K \oplus \langle a \rangle = A \oplus K$ has finite index in $G$, concluding the proof of the theorem. 
\end{proof}

\noindent
A nice application of Theorem \ref{thm:splitting} is the following result, which is a particular of a much more general statement \cite{MR0189027}:

\begin{cor}\label{cor:Euler}
Let $X$ be a compact nonpositively curved cube complex. If the Euler characteristic $\chi(X)$ is non-zero, then the center of $\pi_1(X)$ must be trivial.
\end{cor}

\noindent
Recall that the \emph{Euler characteristic} of a group $G$ is 
$$\chi(G)= \sum\limits_{i \geq 0} (-1)^i~ \mathrm{rank}_\mathbb{Z}H_i(X)$$
where $X$ is a classifying space of $G$. This sum does not depend on the classifying space we choose, but it may not converge. A case where the Euler characteristic is clearly well-defined is when the group admits a classifying space with only finitely many cells. Because CAT(0) cube complexes are contractible, this case includes fundamental groups of compact nonpositively curved cube complexes. 

\begin{proof}[Proof of Corollary \ref{cor:Euler}.]
Assume that $\pi_1(X)$ has a non-trivial center. According to Theorem \ref{thm:splitting}, there exists a finite-index subgroup $H \leq \pi_1(X)$ which decomposes as a product $G=Z \times K$ where $Z$ is the (infinite) center of $\pi_1(X)$ and $K$ some subgroup. Notice that it follows from K\"{u}nneth theorem that $\chi(K)$ is well-defined and that $\chi(G)= \chi(Z) \cdot \chi(K)$. Because $\chi(X)$ has to be a multiple of $\chi(G)$ and that $\chi(Z)=0$, it follows that $\chi(X)$ has to be zero. 
\end{proof}

\subsection{Mapping class groups of surfaces}

\noindent
As mentioned in the introduction, it is an open problem to determine whether or not mapping class groups of surfaces act on CAT(0) cube complexes without global fixed points. Thanks to Theorem \ref{thm:MedianSet}, we are able to show the following restriction by reproducing the proof of \cite[Theorem B]{MR2665003}. 

\begin{thm}\label{thm:MCG}
Let $\Sigma$ be an orientable surface of finite type with genus $\geq 3$. Whenever $\mathrm{Mod}(\Sigma)$ acts on a CAT(0) cube complex, all Dehn twists are elliptic.
\end{thm}

\begin{proof}
Let $X$ be a CAT(0) cube complex on which $\mathrm{Mod}(\Sigma)$ acts and let $\tau \in \mathrm{Mod}(\Sigma)$ be a Dehn twist around a simple closed curve $\gamma \subset \Sigma$. Assume for contradiction that $\tau$ has unbounded orbits. As a consequence of Theorem \ref{thm:MedianSet}, there exists a median subalgebra $Y \subset X$ (namely, the median set of $\tau$) which is left invariant by the centraliser $C(\tau)$ of $\tau$ in $\mathrm{Mod}(\Sigma)$ and which decomposes as a product $T \times F \times Q$ where $F$ is a median flat. Moreover, we know that the action $C(\tau) \curvearrowright Y$ induces an action $C(\tau) \curvearrowright F$. Because $F$ is finite-dimensional, it follows from \cite[Proposition 2.3]{MR2665003} that $\tau$ has infinite order in the abelianisation of $C(\tau)$, contradicting \cite[Proposition 2.4]{MR2665003}.
\end{proof}

\subsection{A cubical flat torus theorem}\label{section:FTT}

\noindent
This section is dedicated to the proof of the following cubical version of the flat torus theorem (compare to \cite[Theorem II.7.20]{MR1744486} in the CAT(0) setting):

\begin{thm}\label{thm:FTT}
Let $G$ be a group acting on a CAT(0) cube complex $X$ and $A \leq G$ a normal finitely generated abelian subgroup. Then there exist a finite-index subgroup $H \leq G$ containing $A$ and a median subalgebra $Y \subset X$ which is $H$-invariant and which decomposes as a product $T \times F \times Q$ of median algebras $T,F,Q$ such that:
\begin{itemize}
	\item $H \curvearrowright Y$ decomposes as a product of actions $H \curvearrowright T,F,Q$;
	\item $F$ is a median flat or a single point; 
	\item $Q$ is a finite-dimensional cube;
	\item $A$ acts trivially on $T$. 
\end{itemize}
\end{thm}

\noindent
We begin with an elementary observation:

\begin{lemma}\label{lem:EllipticSubgroup}
Let $A$ be an abelian group acting on a CAT(0) cube complex $X$. The set of elliptic elements in $A$ defines a subgroup. 
\end{lemma}

\begin{proof}
It suffices to show that, if $a,b \in \mathrm{Isom}(X)$ are two commutating elliptic isometries, then $ab$ is also elliptic. According to Proposition \ref{prop:MedElliptic}, the union $\mathrm{Med}(b)$ of all the cubes of minimal dimension stabilised by $b$ is a median subalgebra. Moreover, it is clearly $\langle a \rangle$-invariant. Because the $\langle a \rangle$-orbits are bounded in $X$, it follows from Lemma \ref{lem:TwoCubulations} that $a$ is also elliptic in $\mathrm{Med}(b)$. Fix a point $x \in \mathrm{Med}(b)$ such that the orbit $\langle a \rangle \cdot x$ is finite. By construction, $\bigcup\limits_{y \in \langle a \rangle \cdot x} \langle b \rangle \cdot y$ is a finite set which is both $\langle a \rangle$- and $\langle b \rangle$-invariant. We conclude that $ab$ has a finite orbit.
\end{proof}

\begin{proof}[Proof of Theorem \ref{thm:FTT}.]
We argue by induction over the rank of $A$, i.e. the minimal cardinality of a generating set. If $A$ is trivial then setting $H=G$, $F=Q=\{\mathrm{vertex} \}$ and $T= X$ leads to the desired conclusion. Now, fixing some $r \geq 0$, assume that our theorem holds for abelian groups of rank $\leq r$ and suppose that $A$ has rank $r+1$. According to Lemma \ref{lem:EllipticSubgroup}, the set $A_0$ of elliptic elements in $A$ defines a subgroup. Moreover, because conjugates of an elliptic isometry remain elliptic, $A_0$ is a normal subgroup in $G$. We distinguish two cases.

\medskip \noindent
First, assume that $A_0$ is non-trivial. According to Proposition \ref{prop:MedElliptic}, there exists a median subalgebra $Y' \subset X$ (namely, the union of all the cubes of minimal dimension stabilised by $A_0$) which is $G$-invariant and which decomposes as a product $Q' \times Z'$ of median algebras $Z',Q'$ such that:
\begin{itemize}
	\item $G \curvearrowright Q' \times Z'$ decomposes as a product $G \curvearrowright Q',Z'$;
	\item $Q'$ is a finite-dimensional cube;
	\item $A_0$ acts trivially on $Z'$.
\end{itemize}
The latter point implies that the action $G \curvearrowright Z'$ factorises through $G/A_0$. Notice that, because $A_0$ is a non-trivial subgroup such that every element in $A$ having a non-zero power in $A_0$ must belong to $A_0$, the rank of the abelian group $A/A_0$ is smaller than the rank of $A$. Consequently, by considering the action of $G/A_0$ on $Z'$ and the normal subgroup $A/A_0$ in $G/A_0$, we know that there exist a finite-index subgroup $K \leq G/A_0$ containing $A/A_0$ and a median subalgebra $Z'' \subset Z'$ which is $K$-invariant and which decomposes as a product $T'' \times F'' \times Q''$ of median algebras $T'',F'',Q''$ such that:
\begin{itemize}
	\item $K \curvearrowright Z''$ decomposes as a product $K \curvearrowright T'',F'',Q''$;
	\item $Q''$ is a finite-dimensional cube and $F''$ is a median flat or a single point;
	\item $A/A_0$ acts trivially on $T''$.
\end{itemize}
Now, let $H$ denote the pre-image of $K$ under the quotient map $G \to G/A_0$. Clearly, $H$ is a finite-index subgroup of $G$ which contains $A$. Also, let $Y$ denote the median algebra $Q' \times Z'' \subset Z'$. Notice that $Y$ decomposes as $T'' \times F'' \times (Q'' \times Q')$ where $Q'' \times Q'$ is finite-dimensional cube. Verifying that $H$ and $Y$ satisfy the conditions of our theorem is straightforward from the construction.

\medskip \noindent
Next, assume that $A_0$ is trivial. In other words, every non-trivial element in $A$ has unbounded orbits. According to Proposition \ref{prop:Centralising}, there exists a finite-index subgroup $K \leq G$ which contains $A$ and which centralises it. Decompose $A$ as $A' \oplus \langle a \rangle$ such that the rank of $A'$ is smaller than the rank of $A$. Notice that, as $K$ centralises $A$, $\langle a \rangle$ is a normal subgroup in $K$. According to Proposition \ref{prop:MinSet}, there exists a median subalgebra $Y'$ (namely, the median set of $a$) which is $K$-invariant and which decomposes as a product $T' \times F' \times Q'$ of median algebras $T',F',Q'$ such that:
\begin{itemize}
	\item $K \curvearrowright Y'$ decomposes as a product $K \curvearrowright T',F',Q'$;
	\item $Q'$ is a finite-dimensional cube and $F'$ is a median flat or a single point;
	\item $a$ acts trivially on $T'$.
\end{itemize}
By considering the action $K \curvearrowright T'$ and the normal subgroup $A' \leq K$, we know that there exist a finite-index subgroup $K' \leq K$ containing $A'$ and a median subalgebra $Y'' \subset Y'$ which is $K'$-invariant and which decomposes as a product $T'' \times F'' \times Q''$ such that:
\begin{itemize}
	\item $K' \curvearrowright Y''$ decomposes as a product $K' \curvearrowright T'',F'',Q''$;
	\item $Q''$ is a finite-dimensional cube and $F''$ is a median flat or a single cube;
	\item $A'$ acts trivially on $T''$.
\end{itemize}
Now, set $H:=\langle K'',a \rangle$ and let $Y$ denote the median subalgebra $Y'' \times F' \times Q' \subset Y'$. Notice that $Y$ decomposes as $T'' \times (F'' \times F') \times (Q'' \times Q')$ where $Q'' \times Q'$ is a finite-dimensional cube and where $F'' \times F'$ is either a single point or a median flat according to Lemma~\ref{lem:ProductFlats}. Verifying that $H$ and $Y$ satisfy the conditions of our theorem is straightforward from the construction. 
\end{proof}

\subsection{Polycyclic groups acting on CAT(0) cube complexes}\label{section:Polycyclic}

\noindent
Recall that the \emph{derived series} of a group $G$ is the decreasing sequence of normal subgroups 
$$G=G_0 \rhd G_1 \rhd \cdots \rhd G_n \rhd \cdots$$
where $G_i$ is the commutator subgroup $[G_{i-1},G_{i-1}]$ of $G_{i-1}$ for every $i \geq 1$. If the series is eventually constant to $\{1\}$, then $G$ is \emph{solvable}, and the smallest $k \geq 0$ such that $G_k= \{1\}$ is referred to as the \emph{derived length} of $G$. A group is \emph{polycyclic} if it is solvable and if all its subgroups are finitely generated.

\medskip \noindent
The main result of this section is the following, which motivates the idea that the action of a polycyclic group on a CAT(0) cube complex essentially factors through a virtually abelian quotient:

\begin{thm}\label{thm:solvable}
Let $G$ be a polycyclic group acting on a CAT(0) cube complex $X$. Then 
\begin{itemize}
	\item $G$ contains a finite-index $H$ such that $$\mathcal{E}= \{ g \in H \mid \text{$g$ is $X$-elliptic} \}$$ defines a normal subgroup of $H$ and such that $H/\mathcal{E}$ is free abelian;
	\item and $H$ stabilises a median subalgebra which is either a median flat or a single point.
\end{itemize}
In particular, $G$ contains a finite-index subgroup which is (locally $X$-elliptic)-by-(free abelian).
\end{thm}

\noindent
We begin by proving the first point of the theorem:

\begin{prop}\label{prop:solvable}
Let $G$ be a polycyclic group acting on a CAT(0) cube complex $X$. Then there exist a finite-index subgroup $H \leq G$ and a median subalgebra $Y \subset X$ which is $H$-invariant and which is a median flat or a single point. 
\end{prop}

\begin{proof}
We argue by induction over the derived length of $G$. If $G$ is trivial, there is nothing to prove; and if $G$ is abelian, the desired conclusion follows from Lemma \ref{lem:AbelianFlat}. So suppose that $G$ has length at least two. Write its derived series
$$\{1\} =G_0 \lhd G_1 \lhd \cdots \lhd G_{n-1} \lhd G_n=G.$$
Notice that $G_1$ is a normal finitely generated abelian subgroup of $G$. It follows from Theorem \ref{thm:FTT} that there exist a finite-index subgroup $K \leq G$ containing $G_1$ and a median subalgebra $Y$ which is $K$-invariant and which decomposes as a product $T \times F \times Q$ such that
\begin{itemize}
	\item the action $K \curvearrowright Y$ decomposes as a product of actions $G \curvearrowright T,F,Q$;
	\item $Q$ is a finite-dimensional cube and $F$ is a median flat or a single point;
	\item $G_1$ acts trivially on $T$.
\end{itemize}
Notice that $K/G_1$ is again polycyclic and that its derived length is smaller than the length of $G$. Therefore, we can apply our induction hypothesis to the action $K/G_1 \curvearrowright T$ in order to find a finite-index subgroup $K' \leq K/G_1$ and a median subalgebra $M \subset T$ which is $K'$-invariant and which is a median flat or a single point. Let $H'$ denote the pre-image of $K'$ under the quotient map $G \to G/G_1$. Then $H'$ is a finite-index subgroup of $G$ which stabilises the median subalgebra $M \times F \times Q$. As a consequence, the kernel $H$ of the action $H' \curvearrowright Q$ stabilises $M \times F \times \{q\}$ for every $q \in Q$, leading to the desired conclusion since $M \times F$ is a median flat or a single point according to Lemma \ref{lem:ProductFlats}. 
\end{proof}

\begin{proof}[Proof of Theorem \ref{thm:solvable}.]
The conclusion is a direct consequence of Propositions \ref{prop:solvable} and~\ref{prop:Busemann}.
\end{proof}

\noindent
The motivation behind Theorem \ref{thm:solvable} was the following question: which solvable groups act properly on CAT(0) cube complexes? As a direct consequence of Theorem \ref{thm:solvable}, we get the following partial answer:

\begin{cor}
Let $G$ be a polycyclic group acting properly on a CAT(0) cube complex. Then $G$ must be virtually free abelian.
\end{cor}

\noindent
It is worth noticing that nevertheless many solvable groups act properly on CAT(0) cube complexes. For instance, for every $n \geq 1$ and every finite group $F$, the lamplighter group $F \wr \mathbb{Z}^n$ acts properly on a CAT(0) cube complex of dimension $2n$ \cite[Proposition~9.33]{Qm}; for every $n,m \geq 1$, the wreath product $\mathbb{Z}^m \wr \mathbb{Z}^n$ acts properly on an infinite-dimensional CAT(0) cube complex \cite{Haagerupwreath, MoiLamp}; and for every $n \geq 1$, the Houghton group $H_n$ acts properly on a CAT(0) cube complex of dimension $n$ \cite{LeeHoughton}, \cite[Example 4.3]{FarleyHughes}.

\medskip \noindent
However, Theorem \ref{thm:solvable} can also be useful to study polycyclic subgroups of groups acting on CAT(0) cube complexes with infinite vertex-stabilisers. In a forthcoming article, we plan to apply this strategy to braided Thompson-like groups. We conclude this section with a remark dedicated to the extension of Proposition \ref{prop:solvable} to more general solvable groups.


\begin{remark}
Let $L_2$ denote the lamplighter group $\left( \bigoplus\limits_{n \in \mathbb{Z}} \mathbb{Z}_2 \right) \rtimes \mathbb{Z}$, where $\mathbb{Z}$ acts on the direct sum by permuting the coordinates. \cite[Proposition 9.33]{Qm} provides a locally finite two-dimensional CAT(0) cube complex $X$ on which $L_2$ acts properly. The median operator can be understood thanks to \cite[Remark 5.9]{MoiLamp}, and it follows that $X$ coincides with the median hull of any $L_2$-orbit. However, $X$ is not a median flat. Consequently, Proposition \ref{prop:solvable} does not hold for the solvable (and even metabelian) group $L_2$. 
\end{remark}

\addcontentsline{toc}{section}{References}

\bibliographystyle{alpha}
{\footnotesize\bibliography{CubicalFTT}}

\end{document}